\documentclass[12pt]{article}
\usepackage[a4paper,mag=1000,includefoot,left=2cm,right=2cm,top=2cm,bottom=2cm,headsep=1cm,footskip=1cm]{geometry}
\usepackage[T2A]{fontenc}
\usepackage[cp1251]{inputenc}
\pdfoutput=1
\usepackage[english]{babel}
\usepackage{amssymb,amsmath,amsfonts,amsthm,indentfirst}
\usepackage{pdflscape}
\usepackage{pgfplotstable,booktabs,colortbl}

\usepackage{fixltx2e}           
\MakeRobust{\eqref}             
\usepackage{cite}
\newtheorem{theorem}{Theorem}

\newtheorem{lemma}{Lemma}
\newtheorem{corollary}{Corollary}
\newtheorem{remark}{Remark}

\DeclareMathOperator*{\cond}{cond}
\renewcommand{\leq}{\leqslant}
\renewcommand{\geq}{\geqslant}
\begin{document}
\begin{center}
\large
\centerline{\Large On fast Fourier solvers for the tensor product}
\smallskip
\centerline{\Large high-order FEM for a generalized Poisson equation}
\vskip 0.3cm

\large{\ Alexander Zlotnik}
\footnote{Corresponding author. Department of Mathematics at Faculty of Economic Sciences,
National Research University Higher School of Economics,
Myasnitskaya 20, 101000 Moscow, Russia.
E-mail: \text{azlotnik2008@gmail.com}}
{ and Ilya Zlotnik}
\footnote{Settlement Depository Company, 2-oi Verkhnii Mikhailovskii proezd 9, building 2, 115419 Moscow, Russia.
E-mail: ilya.zlotnik@gmail.com}
\end{center}
\begin{abstract}
We present direct logarithmically optimal in theory and fast in practice algorithms to implement the tensor product high order finite element method on multi-dimensional rectangular parallelepipeds for solving PDEs of the Poisson kind.
They are based on the well-known Fourier approaches.
The key new points are the fast direct and inverse FFT-based algorithms for expansion in eigenvectors of the 1D eigenvalue problems for the high order FEM.
The algorithms can further be used for numerous applications, in particular, to implement the tensor product high order finite element methods for various time-dependent PDEs.
Results of numerical experiments in 2D and 3D cases are presented.
\end{abstract}

\section{\large Introduction}

We present direct fast algorithms to implement $n$th order ($n\geq 2$) finite element method (FEM) on rectangular parallelepipeds \cite{C02} for solving a $N$-dimensional generalized Poisson equation ($N\geq 2$) with the Dirichlet boundary condition.
The algorithms are based on the well-known Fourier approaches, e.g., see \cite{BFK11,KS07,SN78,S77} and references therein.
The key new points are the fast direct and inverse algorithms for expansion in eigenvectors of the 1D eigenvalue problems for the high order FEM utilizing several versions of the discrete fast Fourier transform (FFT) \cite{BRY07}.
This solves the old known problem, see \cite[p. 271]{BFK11}, and makes the full algorithms logarithmically optimal with respect to the number of elements as in the case of the bilinear elements ($n=1$) or standard finite-difference schemes.
The algorithms are fast in practice (faster than the theoretical expectations) and demonstrate only a mild growth in $n$ starting from the standard case $n=1$.
For example, in the 9th order case,
the 2D FEM system for $2^{20}$ elements containing almost $85\cdot10^6$ unknowns and
the 3D FEM system for $2^{18}$ elements containing more than $190\cdot10^6$ unknowns are solved respectively in less than 2 and 15 min on an ordinary laptop using Matlab R2016a code, see details below.
\par The algorithms can further serve for a variety of applications including general 2nd order elliptic equations (as preconditioners), for  the $N$-dimensional heat, wave or time-dependent Schr\"{o}dinger PDEs, etc.
They can be applied for some non-rectangular domains, in particular, by involving meshes topologically equivalent to rectangular ones \cite{D96}.
Other standard boundary conditions can be covered as well, see a brief description in \cite{ZZDAN16}.
Moreover, the Fourier structure of algorithms is especially valuable for solving some wave physics problems, in particular, involving non-local boundary conditions, e.g., see \cite{BFK11,DZZ14,ZZ12}, whence our own interest arose.
Clearly the algorithms are also highly parallelizable.
\par The paper is organized as follows.
In Section \ref{sect2}, the statement of the 1D FEM eigenvalue problem together with auxiliary FEM eigenvalue problems on and inside the reference element are given.
The basic Section \ref{sect4} is devoted to a description of eigenvalues and eigenvectors of the 1D FEM eigenvalue problem and the fast direct and inverse algorithms for expansion in these eigenvectors.
Applications to the generalized Poisson equation in a $N$-dimensional rectangular parallelepiped with the Dirichlet boundary condition are described in Section \ref{sect5}.
Results of numerical experiments for $N=2$ and 3 are presented in detail in Section \ref{sect6};
all of them include the standard case $n=1$ for comparison.
\section{\large The statement of 1D FEM eigenvalue problem}
\label{sect2}
\setcounter{equation}{0}
\setcounter{proposition}{0}
\setcounter{theorem}{0}
\setcounter{lemma}{0}
\setcounter{corollary}{0}
\setcounter{remark}{0}
We first consider in detail the FEM for the simplest 1D eigenvalue ODE problem
\begin{gather}
 -u''(x)=\lambda u(x)\ \ \text{on}\ \ [0,X],\ \ u(0)=u(X)=0,\ \ u(x)\not\equiv 0.
\label{eq:diff eig pr}
\end{gather}
\par We take the uniform mesh $\bar{\omega}_h$ with the nodes $x_j=jh$, $j=\overline{0,K}$ (i.e., $0\leqslant j\leqslant K$)
and the step $h=X/K$.
Let $H_h^{(n)}[0,X]$ be the FEM space of the piecewise-polynomial functions $\varphi\in C[0,X]$ such that
$\varphi(x)\in \mathcal{P}_n$ for $x\in [x_{j-1}, x_j]$, $j=\overline{1,K}$, with $\varphi(0)=\varphi(X)=0$;
here $\mathcal{P}_n$ is the space of polynomials having at most $n$th degree, $n\geq 2$.
\par Let $S_K^{(n)}$ be the space of vector functions $w$ such that
$w_j\in \mathbb{R}$ for $j=\overline{0,K}$ with $w_0=w_K=0$ and $w_{j-1/2}\in\mathbb{R}^{n-1}$, $j=\overline{1,K}$.
Clearly $\dim S_K^{(n)}=nK-1$.
A function $\varphi\in H_h^{(n)}[0,X]$ is uniquely defined by its values at the mesh nodes $\varphi_j=\varphi(x_j)$, $j=\overline{0,K}$, with $\varphi_0=\varphi_K=0$, and inside the elements $\varphi_{j-1/2}=\{\varphi(x_{j-1}+(l/n)h)\}_{l=1}^{n-1}$, $j=\overline{1,K}$, that form the element in $S_K^{(n)}$.
\par We use the following scaled operator form of the standard FEM discretization for problem \eqref{eq:diff eig pr}
\begin{gather}
 \mathcal{A}v=\lambda\mathcal{C}v,\ \ v\in S_K^{(n)},\ \ v\neq 0.
\label{eq:eig_glob}
\end{gather}
Here $\mathcal{A}=\mathcal{A}^T>0$ and $\mathcal{C}=\mathcal{C}^T>0$ are the global (scaled) stiffness and mass operators (matrices) acting in $S_K^{(n)}$ and together with $\lambda$ \textit{independent on} $h$; the true approximate eigenvalues are $\lambda_h=4h^{-2}\lambda$.
\par Let $A=\{A_{kl}\}_{k,l=0}^n$ and
$C=\{C_{kl}\}_{k,l=0}^n$ be the local stiffness and mass matrices related to the reference element $\sigma_0=[-1,1]$
with the following entries
\[
 A_{kl} = \int\nolimits_{\sigma_0}e'_k(x)e'_l(x)\,dx,\ \ C_{kl} = \int\nolimits_{\sigma_0}e_k(x)e_l(x)\,dx,
\]
where $\{e_l\}_{l=0}^n$ is the Lagrange basis in $\mathcal{P}_n$ such that
$e_l\big(-1+(2k)/n\big)=\delta_{kl}$, for $k,l=\overline{0,n}$,
and $\delta_{kl}$ is the Kronecker delta.
The matrices $A$, $C$ and the related matrix pencil $G(\lambda):=A-\lambda C$ have the following $3\times3$--block form
\begin{equation}
 A=
\left(\hspace{-4pt}
\begin{array}{ccc}
 a_0& a^T           & a_n\\
 a  & \widetilde{A} &\check{a}\\
 a_n& \check{a}^T   &a_0
\end{array}
\hspace{-4pt}\right),\
 C=
\left(\hspace{-4pt}
\begin{array}{ccc}
 c_0& c^T           & c_n\\
 c  & \widetilde{C} &\check{c}\\
 c_n& \check{c}^T   &c_0
\end{array}
\hspace{-4pt}\right),\
 G(\lambda)=\left(\hspace{-4pt}
\begin{array}{ccc}
 g_0(\lambda)& g^T(\lambda)          & g_n(\lambda)\\
 g(\lambda)  & \widetilde{G}(\lambda)&\check{g}(\lambda)\\
 g_n(\lambda)& \check{g}^T(\lambda)  &g_0(\lambda)
\end{array}
\hspace{-4pt}\right).
\label{eq:matrAC}
\end{equation}
Here $\widetilde{A}$, $\widetilde{C}$ and $\widetilde{G}(\lambda)=\widetilde{A}-\lambda \widetilde{C}$ are square matrices of order $n-1$ with the column vectors $a,c,g(\lambda)=a-\lambda c\in\mathbb{R}^{n-1}$
whereas $\check{p}_l\equiv (Pp)_l=p_{n-l}$, $l=\overline{1,n-1}$, for $p\in\mathbb{R}^{n-1}$.
The matrices $A$, $C$ and $G(\lambda)$ are \textit{bisymmetric} (i.e. symmetric with respect to the main and secondary diagonals).
\par Notice that $P_{ij}=\delta_{i(n-j)}$ and $P^T=P^{-1}=P$.
Let $\mathbb{R}_e^{n-1}$ and $\mathbb{R}_o^{n-1}$ be the subspaces of even and odd vectors in $\mathbb{R}^{n-1}$, i.e. such that respectively $Pp=p$ and $Pp=-p$.
The decomposition $\mathbb{R}^{n-1}=\mathbb{R}_e^{n-1}\oplus\mathbb{R}_o^{n-1}$ (for $n\geq 3$) is implemented by the formulas
\begin{gather}
p=p_e+p_o,\ \ p_e:=0.5(p+\check{p}),\ \ p_o:=0.5(p-\check{p}).
\label{eq:eodecomp}
\end{gather}
Notice that $\dim\mathbb{R}_e^{n-1}=[n/2]$ and $\dim\mathbb{R}_o^{n-1}=[(n-1)/2]$, with $\mathbb{R}_o^{n-1}=\{0\}$ for $n=2$.
Clearly
\begin{gather}
 \check{p}\cdot q=p\cdot\check{q},\ \ \check{p}\cdot\check{q}=p\cdot q\ \ \text{for any}\ \ p,q\in\mathbb{R}^{n-1}.
\label{eq:eodprop1}
\end{gather}
Hereafter the symbol $\cdot$ denotes the inner product of vectors in $\mathbb{R}^{n-1}$.

\par Then problem \eqref{eq:eig_glob} can be written in the following explicit form
\begin{gather}
 g_n(\lambda) v_{j-1}+\check{g}(\lambda)\cdot v_{j-1/2}+2g_0(\lambda) v_j+g(\lambda)\cdot v_{j+1/2}+g_n(\lambda) v_{j+1}=0,\ \
 j=\overline{1,K-1},
\label{eq:eigpr1}
\\[1mm]
 g(\lambda)v_{j-1}+\widetilde{G}(\lambda)v_{j-1/2}+\check{g}(\lambda)v_j=0,\ \ j=\overline{1,K},
\label{eq:eigpr2}
\end{gather}
with $v_0=v_K=0$, $v\not\equiv 0$.

\par We also consider the auxiliary eigenvalue problems on and inside the reference element $\sigma_0$
\begin{gather}
 Ae=\lambda Ce,\ \ e\in\mathbb{R}^{n+1},\ \ e\neq 0;
\label{eq:eig_elem}\\[1mm]
 \widetilde{A}e=\lambda\widetilde{C}e,\ \ e\in\mathbb{R}^{n-1},\ \ e\neq 0,
\label{eq:eig_elemt}
\end{gather}
where clearly $A\geq 0$, $C>0$ and $\widetilde{A}=\widetilde{A}^T>0$, $\widetilde{C}=\widetilde{C}^T>0$; see some their properties in  \cite{ZZ12} (where the problem similar to \eqref{eq:eigpr1}, \eqref{eq:eigpr2} on the uniform mesh on $[0,\infty)$ for $\lambda\in \mathbb{C}$ was studied).
Denote by $S_n$ and $\tilde{S}_n$ their spectra.
Let $\{\lambda_0^{(l)},e^{(l)}\}_{l=1}^{n-1}$ be eigenpairs of problem \eqref{eq:eig_elemt}.
\begin{lemma}
\label{prop1}
1. The subspaces $\mathbb{R}_e^{n-1}$ and $\mathbb{R}_e^{n-1}$ are invariant with respect to $\widetilde{A}$ and $\widetilde{C}$.
Thus each eigenvector in $\{e^{(l)}\}_{l=1}^{n-1}$ can be chosen either even or odd.
Also $\lambda_0^{(l)}>0$, $l=\overline{1,n-1}$.
\par 2. Similar properties are valid for problem \eqref{eq:eig_elem} with the exception of one simple zero eigenvalue.
\end{lemma}
\begin{proof}
Any bisymmetric matrix $B$ of the order $n-1$ commutes with $P$, i.e.
\begin{gather}
 BP=PB,
\label{eq:bisym}
\end{gather}
that implies the main result of Item 1.
The property $\lambda_0^{(l)}>0$, $l=\overline{1,n-1}$
is well-known.
\par For Item 2, the argument is similar taking into account that $A(1\ldots 1)^T=0$ (concerning simplicity of $\lambda=0$, see Proposition 5 in \cite{ZZ12}).
\end{proof}

One can check by the direct computation that all the eigenvalues in $S_n$ and $\tilde{S}_n$ are simple and $S_n\cap\tilde{S}_n=\emptyset$ at least for $1\leqslant n\leqslant 9$, see \cite{ZZ12}.

For low $n$, one can find $S_n$ and $\tilde{S}_n$ analytically (with the help of Mathematica), in particular,
$\tilde{S}_2=\{2.5\}$,
$\tilde{S}_3=\{2.5,10.5\}$,
$\tilde{S}_4=\{14\pm\sqrt{133},10.5\}$ and
$\tilde{S}_5=\{14\pm\sqrt{133},30\pm9\sqrt{5}\}$ (repeatability of the eigenvalues is not occasional, see \cite{ZZ12}).

\par We choose $\{e^{(l)}\}_{l=1}^{n-1}$ as in Lemma \ref{prop1} using scaling $\widetilde{C}e^{(l)}\cdot e^{(l)}=1$.
\begin{lemma}
\label{lem1}
Let $\widetilde{G}(\lambda)p=-g(\lambda)$, where  $\lambda\not\in\tilde{S}_n$.
Let the vectors $a$ and $c$ be expanded as
\begin{gather}
 a=\sum_{l=1}^{n-1}a^{(l)}\widetilde{C}e^{(l)},\ \ c=\sum_{l=1}^{n-1}c^{(l)}\widetilde{C}e^{(l)},\ \ \text{with}\ \
 a^{(l)}=a\cdot e^{(l)},\ \ c^{(l)}=c\cdot e^{(l)}.
\label{eq:formpac}
\end{gather}
See $\widetilde{G}(\lambda)$, $g(\lambda)$, $a$ and $c$ in \eqref{eq:matrAC}.
Then the following formulas hold
\begin{gather*}
 p=\sum_{l=1}^{n-1}\frac{a^{(l)}-\lambda c^{(l)}}{\lambda-\lambda_0^{(l)}}e^{(l)}
  =\sum_{l=1}^{n-1}\frac{a^{(l)}-\lambda_0^{(l)} c^{(l)}}{\lambda-\lambda_0^{(l)}}e^{(l)}-\widetilde{C}^{-1}c.
\label{eq:formp}
\end{gather*}
\end{lemma}
\begin{proof}
For the expansions $p=\sum_{l=1}^{n-1}p_le^{(l)}$ and \eqref{eq:formpac}, we have
\[
 \widetilde{G}(\lambda)p=\sum_{l=1}^{n-1}\big(\lambda_0^{(l)}-\lambda\big)p_l\widetilde{C}e^{(l)},\ \
 g(\lambda)=\sum_{l=1}^{n-1}\big(a^{(l)}-\lambda c^{(l)}\big)\widetilde{C}e^{(l)},
\]
and the result easily follows.
\end{proof}
\section{\large Solving of the 1D FEM eigenvalue problem and related FFT-based algorithms}
\label{sect4}
\setcounter{equation}{0}
\setcounter{proposition}{0}
\setcounter{theorem}{0}
\setcounter{lemma}{0}
\setcounter{corollary}{0}
\setcounter{remark}{0}
Below we impose the following assumption:
\[
 (A)\ \ \text{the eigenvalues in}\ S_n\ \text{and}\ \tilde{S}_n\ \text{are simple and}\ S_n\cap\tilde{S}_n=\emptyset.
\]
Recall that it is valid at least for $2\leq n\leq 9$ (below in Section \ref{sect6} we verify it up to $n=21$).
\par We introduce the auxiliary equation
\begin{gather}
 \widehat{\gamma}(\lambda)\equiv -(g_0-g\cdot\widetilde{G}^{-1}g)(\lambda)/(g_n-\check{g}\cdot\widetilde{G}^{-1}g)(\lambda)=\theta,
\label{eq:theta0}
\end{gather}
where  $\lambda\not\in\tilde{S}_n$, with the parameter $\theta$.
Its solving is equivalent to finding the roots of a polynomial having at most $n$th degree, see \cite{ZZ12}.
In particular, owing to Lemma \ref{lem1} this equation can be rewritten as
\begin{gather}
 a_0-\lambda c_0+\sum_{l=1}^{n-1}\frac{(a^{(l)}-\lambda c^{(l)})^2}{\lambda-\lambda_0^{(l)}}
 =-\theta\Big(a_n-\lambda c_n+\sum_{l=1}^{n-1}\frac{(\check{a}^{(l)}-\lambda \check{c}^{(l)})(a^{(l)}-\lambda c^{(l)})}{\lambda-\lambda_0^{(l)}}\Big).
\label{eq:theta}
\end{gather}
Here
$\check{a}^{(l)}=\check{a}\cdot e^{(l)}$ and $\check{c}^{(l)}=\check{c}\cdot e^{(l)}$.
Moreover, for $2\leqslant n\leqslant 9$ computations help to confirm that the vectors $e^{(l)}$ are even and odd respectively for odd and even $l$ provided that $\lambda_0^{(1)}<\ldots<\lambda_0^{(n-1)}$; therefore $\check{a}^{(l)}=(-1)^la^{(l)}$ and $\check{c}^{(l)}=(-1)^lc^{(l)}$, $l=\overline{1,n-1}$.
\par We define the simplest inner product in $S_K^{(n)}$ and the corresponding squared $\mathcal{C}$-norm
\[
 (y,v)_{S_K^{(n)}}:=\sum_{j=1}^{K-1}y_jv_j+\sum_{j=1}^{K}y_{j-1/2}\cdot v_{j-1/2},\ \
 \|v\|_{\mathcal{C}}^2:=(\mathcal{C}v,v)_{S_K^{(n)}}.
\]
\par Next theorem describes eigenvalues and eigenvectors of problem \eqref{eq:eig_glob}.
\begin{theorem}
\label{th:eigpares}
1. The spectrum of problem \eqref{eq:eig_glob} consists in $\tilde{S}_n$
and the numbers $\big\{\lambda_k^{(l)}\big\}_{l=1}^n$ that are all $n$ (and all positive real) solutions to equation \eqref{eq:theta} with $\theta=\theta_k:=\cos\frac{\pi k}{K}$ for $k=\overline{1,K-1}$.
The numbers $\big\{\lambda_k^{(l)}\big\}_{l=1}^n$ differ from $\{\lambda_0^{(l)}\}_{l=1}^{n-1}$ and are different for fixed $k$ .
\par 2. The following eigenvector corresponds to the eigenvalue $\lambda_0^{(l)}$:
\begin{gather}
 s_{0,j}^{(l)}=0,\,\ j=\overline{1,K-1},\ \
 s_{0,j-1/2}^{(l)}=(-P)^{j-1}e^{(l)},\,\ j=\overline{1,K},
\label{eq:eig vec1}
\end{gather}
for $l=\overline{1,n-1}$. Here $(-P)^{j-1}e=(-1)^{j-1}e$ for even $e$, $(-P)^{j-1}e=e$ for odd $e$.
\par 3. The following eigenvector corresponds to the eigenvalue $\lambda_k^{(l)}$:
\begin{gather}
s_{k,j}^{(l)}=s_{k,j},\,\ j=\overline{1,K-1},\ \
 s_{k,j-1/2}^{(l)}=p_k^{(l)}\sin\frac{\pi k(j-1)}{K}+\check{p}_k^{(l)}\sin\frac{\pi kj}{K},\,\ j=\overline{1,K},
\label{eq:eig vec2}
\end{gather}
where $s_{k,j}:=\sin\frac{\pi kj}{K}$ and $p_k^{(l)}\in\mathbb{R}^{n-1}$ solves the non-degenerate algebraic system
$\widetilde{G}\big(\lambda^{(l)}_k\big)p_k^{(l)}=-g\big(\lambda^{(l)}_k\big)$, for $k=\overline{1,K-1}$, $l=\overline{1,n}$.

\par 4. The introduced eigenvectors are $\mathcal{C}$-orthogonal, i.e.
\begin{gather}
 \big(\mathcal{C}s_{k}^{(l)},s_{\tilde{k}}^{(\tilde{l})}\big)_{S_K^{(n)}}=0
\label{eq:ort}
\end{gather}
 for any $k,\tilde{k}\in\overline{0,K-1}$, $l\in\overline{1,n-\delta_{k0}}$ and $\tilde{l}\in\overline{1,n-\delta_{\tilde{k}0}}$ such that $k\neq\tilde{k}$ and/or $l\neq\tilde{l}$.
\par Consequently they form the basis in $S_K^{(n)}$, i.e. any $w\in S_K^{(n)}$ can be uniquely expanded as
\begin{gather}
 w=\sum_{l=1}^{n-1}w_{0l}s_0^{(l)}+\sum_{k=1}^{K-1}\sum_{l=1}^n w_{kl}s_k^{(l)}.
\label{eq:decomp}
\end{gather}
\end{theorem}
\begin{proof}
\par 1. We distinguish between two cases.
Let first $\lambda\in \tilde{S}_n$ and $e$ satisfy \eqref{eq:eig_elemt}.
Then, for any $j=\overline{1,K}$, using equation \eqref{eq:eigpr2} we get
\[
 0=v_{j-1/2}\cdot\widetilde{G}(\lambda)e=\widetilde{G}(\lambda)v_{j-1/2}\cdot e
 =-\big[(g(\lambda)\cdot e)v_{j-1}+(\check{g}(\lambda)\cdot e)v_j\big].
\]
Clearly
\begin{gather*}
 G(\lambda)\left(
\begin{array}{c}
 0\\[1mm]
 e\\[1mm]
 0
\end{array}
\right)=
\left(
\begin{array}{c}
 g(\lambda)\cdot e\\[1mm]
 0\\[1mm]
 \check{g}(\lambda)\cdot e
\end{array}
\right).
\end{gather*}
Since $\lambda\not\in S_n$ by assumption (A), we have $g(\lambda)\cdot e\neq 0$ or $\check{g}(\lambda)\cdot e\neq 0$.
Owing to assumption (A) and Lemma \ref{prop1}, $\lambda$ is simple in $\tilde{S}_n$ and $e$ is either even or odd.
Correspondingly either
$\check{g}(\lambda)\cdot e=g(\lambda)\cdot e\neq 0$ and
$v_j=-v_{j-1}$, or
$\check{g}(\lambda)\cdot e=-g(\lambda)\cdot e\neq 0$ and
$v_j=v_{j-1}$.
Since $v_0=0$, in both cases we get
\begin{gather}
 v_j=0,\ \ j=\overline{0,K}.
\label{eq:s3t1f1}
\end{gather}
Thus equation \eqref{eq:eigpr2} is reduced to $\widetilde{G}v_{j-1/2}=0$ and implies that $v_{j-1/2}=c_{j-1/2}e$.
\par Now equation \eqref{eq:eigpr1} is reduced to
\[
 c_{j-1/2}\check{g}(\lambda)\cdot e+c_{j+1/2}g(\lambda)\cdot e=0,\ \ j=\overline{1,K-1}.
\]
Therefore $c_{j+1/2}=-c_{j-1/2}$ for even $e$ or $c_{j+1/2}=c_{j-1/2}$ for odd $e$.
Consequently the sought eigenvector $v$ satisfies \eqref{eq:s3t1f1} together with
\[
 v_{j-1/2}=(-1)^{j-1}e\ \ \text{for even}\ \ e,\ \  v_{j-1/2}=e\ \ \text{for odd}\ \ e,\ \ j=\overline{1,K}
\]
($v$ is defined up to a non-zero constant multiplier). Thus we come to eigenvectors \eqref{eq:eig vec1}.
\par 2. Next let $\lambda\not\in \tilde{S}_n$. Then from equation \eqref{eq:eigpr2} we get
\begin{gather}
 v_{j-1/2}=-G^{-1}(\lambda)[v_{j-1}g(\lambda)+v_j\check{g}(\lambda)],\ \ j=\overline{1,K}.
\label{eq:s3t1f3}
\end{gather}
Inserting this into equation \eqref{eq:eigpr1}, we find the three-point equation
\begin{gather}
 \hat{g}_n(\lambda)v_{j-1}+2\hat{g}_0(\lambda)v_j+\hat{g}_n(\lambda)v_{j+1}=0,\ \ j=\overline{1,K-1},
\label{eq:s3t1f5}
\end{gather}
where
\[
 \hat{g}_0(\lambda)=(g_0-g\cdot\widetilde{G}^{-1}g)(\lambda), \ \ \hat{g}_n(\lambda)=(g_n-\check{g}\cdot\widetilde{G}^{-1}g)(\lambda).
\]
\par Straightforwardly the following equalities hold
\begin{gather*}
 G(\lambda)\left(
\begin{array}{c}
 1\\[1mm]
 -(\widetilde{G}^{-1}g)(\lambda)\\[1mm]
 0
\end{array}
\right)=
\left(
\begin{array}{c}
 \hat{g}_0(\lambda)\\[1mm]
 0\\[1mm]
 \hat{g}_n(\lambda)
\end{array}
\right),\ \
 G(\lambda)\left(
\begin{array}{c}
 0\\[1mm]
 -(\widetilde{G}^{-1}\check{g})(\lambda)\\[1mm]
 1
\end{array}
\right)=
\left(
\begin{array}{c}
 \hat{g}_n(\lambda)\\[1mm]
 0\\[1mm]
 \hat{g}_0(\lambda)
\end{array}
\right).
\label{eq:s3t1f7}
\end{gather*}
If $\hat{g}_n(\lambda)=\hat{g}_0(\lambda)=0$, then the equalities mean that $\lambda$ is at least double eigenvalue for problem \eqref{eq:eig_elem} that contradicts assumption (A).
\par If $\hat{g}_n(\lambda)=0$ and $\hat{g}_0(\lambda)\neq 0$, then equation \eqref{eq:s3t1f5} together with \eqref{eq:s3t1f3} lead
to $v=0$ thus such $\lambda$ does not satisfy \eqref{eq:eig_glob}.
\par Therefore $\hat{g}_n(\lambda)\neq 0$ and equation \eqref{eq:s3t1f5} is simplified to
\begin{gather}
 v_{j-1}-2\hat{\gamma}(\lambda)v_j+v_{j+1}=0,\ \ j=\overline{1,K-1},
\label{eq:s3t1f9}
\end{gather}
with the function $\hat{\gamma}(\lambda)=\hat{g}_0(\lambda)/\hat{g}_n(\lambda)$ (see it also in \eqref{eq:theta0}).
Since $v_0=v_n=0$, we can use the expansion
\[
 v_j=\sum_{k=1}^{K-1}\tilde{v}_ks_{k,j},\ \ j=\overline{0,K},
\]
and define the vector $\tilde{{\bf v}}:=(\tilde{v}_1,\ldots,\tilde{v}_{K-1})$ of its coefficients.
Using the expansion in \eqref{eq:s3t1f9} gives
\begin{gather}
 2\sum_{k=1}^{K-1}\tilde{v}_k\big(\theta_k-\hat{\gamma}(\lambda)\big)s_{k,j}=0,\ \ j=\overline{1,K-1}.
\label{eq:s3t1f11}
\end{gather}
Clearly this equality is valid for some $\tilde{{\bf v}}\neq 0$
if and only if
\begin{gather}
 \hat{\gamma}(\lambda)=\theta_k\ \  \text{for some}\ \ k=\overline{1,K-1}.
\label{eq:s3t1f15}
\end{gather}
\par Notice that $\tilde{{\bf v}}=0$ is equivalent to $v=0$ in $S_K^{(n)}$ (taking into account formula \eqref{eq:s3t1f3}).
Therefore $\lambda$ satisfies \eqref{eq:s3t1f15}; moreover, $\tilde{v}_j=\delta_{kj}$ and consequently $v_j=s_{k,j}$, $j=\overline{0,K}$, together with
\[
 v_{j-1/2}=-s_{k,j-1}(G^{-1}g)(\lambda)-s_{k,j}(G^{-1}\check{g})(\lambda),\ \ j=\overline{1,K},
\]
see \eqref{eq:s3t1f3} (all last three equalities are valid up to the same non-zero multilplier).
Thus we come to eigenvectors \eqref{eq:eig vec2}.
\par The total amount of eigenvalues $\lambda\not\in \tilde{S}_n$ (taking into account their possible multiplicity) is
$\dim S_K^{(n)}-(n-1)=n(K-1)$.
The maximal amount of roots algebraic equations \eqref{eq:s3t1f15} for all $k$ is the same so that each equation \eqref{eq:s3t1f15} has to possess exactly $n$ distinct roots (for fixed $k$, the written eigenvector $v$ is defined by $\lambda$ uniquely).
\par 3. Property \eqref{eq:ort} is knowingly valid for eigenvectors $s_{k}^{(l)}$ and $s_{\tilde{k}}^{(\tilde{l})}$ corresponding to different eigenvalues of problem \eqref{eq:eig_glob}, in particular, for $k=0$ and $\tilde{k}\neq 0$, or $k=\tilde{k}$ and $l\neq\tilde{l}$.
The remaining case will be covered below in Corollary \ref{cor:lem:scalar_prod1} of the related Lemma \ref{lem:scalar_prod1}.
\end{proof}
\par Notice that:
(1) the vectors $s_0^{(l)}$ are used only to describe the algorithm, and only the vectors $e^{(l)}$ are applied in its implementation;
(2) $s_{k,j}^{(l)}$ are independent on $l$;
(3) the vectors $p_k^{(l)}$ are independent on $j$ and can also be computed owing to Lemma \ref{lem1}.
\begin{lemma}
\label{lem:scalar_prod1}
Let $w\in S_K^{(n)}$ and $w_{j-1/2}=qw_{j-1}+\check{q}w_j$, $j=\overline{1,K}$, for some $q\in\mathbb{R}^{n-1}$.
Then
\begin{gather}
 \big(\mathcal{C}s_0^{(l)},w\big)_{S_K^{(n)}}=0,\ \ l=\overline{1,n-1}.
\label{sp11}\\
 \big(\mathcal{C}s_k^{(l)},w\big)_{S_K^{(n)}}
\nonumber\\
 =2\big\{c_0+c\cdot p_k^{(l)}+\big(\widetilde{C}p_k^{(l)}+c\big)\cdot q
 +\theta_k\big[c_n+c\cdot\check{p}_k^{(l)}+\big(\widetilde{C}p_k^{(l)}+c\big)\cdot \check{q}\big]\big\}(s_k,w)_{\omega_h},
\label{sp1}
\end{gather}
for $k=\overline{1,K-1}$, $l=\overline{1,n}$,  where $(s_k,w)_{\omega_h}:=\sum_{j=1}^{K-1}s_{k,j}w_j$.
\end{lemma}
\begin{proof}
1. For any $v,w\in S_K^{(n)}$, recalling notation \eqref{eq:matrAC} we have
\begin{gather}
 (\mathcal{C}v,w)_{S_K^{(n)}}
 =\sum_{j=1}^{K-1}\big[2c_0v_j+c_n(v_{j-1}+v_{j+1})+\check{c}\cdot v_{j-1/2}+c\cdot v_{j+1/2}\big]w_j
\nonumber\\[1mm]
 +\sum_{j=1}^K\big(cv_{j-1}+\widetilde{C}v_{j-1/2}+\check{c}v_{j+1}\big)\cdot w_{j-1/2}.
\label{sp2}
\end{gather}

\par 2. According to formulas \eqref{sp2} and \eqref{eq:eig vec1}, for even $e^{(l)}$, we get
\begin{gather*}
 \big(\mathcal{C}s_0^{(l)},w\big)_{S_K^{(n)}}
 =\sum_{j=1}^{K-1}(-1)^j(c-\check{c})\cdot e^{(l)} w_j
 +\sum_{j=1}^K(-1)^{j-1}\widetilde{C}e^{(l)}\cdot (qw_{j-1}+\check{q}w_j)
\\[1mm]
 =\widetilde{C}e^{(l)}\cdot(q-\check{q})\sum_{j=1}^{K-1}(-1)^j w_j=0
\end{gather*}
since $(c-\check{c})\cdot e^{(l)}=0$ and $\widetilde{C}e^{(l)}\cdot(q-\check{q})=0$ for any $c,q\in\mathbb{R}^{n-1}$ as well as  $w_0=w_K=0$.
\par For odd $e^{(l)}$, we similarly get
\begin{gather*}
 \big(\mathcal{C}s_0^{(l)},w\big)_{S_K^{(n)}}
 =\sum_{j=1}^{K-1}(c+\check{c})\cdot e^{(l)} w_j
 +\sum_{j=1}^K\widetilde{C}e^{(l)}\cdot (qw_{j-1}+\check{q}w_j)
 =\widetilde{C}e^{(l)}\cdot(q+\check{q})\sum_{j=1}^{K-1}w_j=0
\end{gather*}
since $(c+\check{c})\cdot e^{(l)}=0$ and $\widetilde{C}e^{(l)}\cdot(q+\check{q})=0$ for any $c,q\in\mathbb{R}^{n-1}$ as well as $w_0=w_K=0$.
Equality \eqref{sp11} is proved.

\par 3. Formula \eqref{sp2} together with \eqref{eq:eig vec2} and \eqref{eq:eodprop1} imply that
\begin{gather}
 \big(\mathcal{C}s_k^{(l)},w\big)_{S_K^{(n)}}
 =\sum_{j=1}^{K-1}\big[2\big(c_0+c\cdot p_k^{(l)}\big)s_{k,j}+\big(c_n+c\cdot \check{p}_k^{(l)}\big)(s_{k,j-1}+s_{k,j+1})\big]w_j
\nonumber\\[1mm]
 +\sum_{j=1}^K\big[\big(\widetilde{C}p_k^{(l)}+c\big)s_{k,j-1}+\big(\widetilde{C}\check{p}_k^{(l)}+\check{c}\big)s_{k,j}\big]\cdot w_{j-1/2}
 =:\sum\nolimits^{'}+\sum\nolimits^{''}.
\label{sp4}
\end{gather}
\par Owing to formula
\begin{gather}
 s_{k,j-1}+s_{k,j+1}=2\theta_ks_{k,j}
\label{sp5}
\end{gather}
we first derive
\begin{gather}
 \sum\nolimits^{'}=2\big[c_0+c\cdot p_k^{(l)}+\theta_k(c_n+c\cdot \check{p}_k^{(l)})\big](s_k,w)_{\omega_h}.
\label{sp7}
\end{gather}
Second, using formulas \eqref{eq:eodprop1} and \eqref{eq:bisym}, we get
\[
 \sum\nolimits^{''}=\sum_{j=1}^K\big(\widetilde{C}p_k^{(l)}+c\big)\cdot q\,(s_{k,j-1}w_{j-1}+s_{k,j}w_j)
 +\big(\widetilde{C}p_k^{(l)}+c\big)\cdot\check{q}\,(s_{k,j-1}w_j+s_{k,j}w_{j-1}).
\]
Owing to $s_{k,0}=s_{k,K}=0$, $w_0=w_K=0$ as well as formulas \eqref{sp5},
we further derive
\begin{gather}
 \sum\nolimits^{''}=2\big(\widetilde{C}p_k^{(l)}+c\big)\cdot q\,(s_k,w)_{\omega_h}
 +\big(\widetilde{C}p_k^{(l)}+c\big)\cdot\check{q}\sum_{j=1}^{K-1}(s_{k,j-1}+s_{k,j+1})w_j
\nonumber\\[1mm]
 =2\big[\big(\widetilde{C}p_k^{(l)}+c\big)\cdot q+\theta_k\big(\widetilde{C}p_k^{(l)}+c\big)\cdot\check{q}\big](s_k,w)_{\omega_h}.
\label{sp9}
\end{gather}
Adding \eqref{sp7} and \eqref{sp9}, we prove \eqref{sp1}.
\end{proof}
\begin{corollary}
\label{cor:lem:scalar_prod1}
The orthogonality property \eqref{eq:ort} from Theorem \ref{th:eigpares}, Item 4 is valid.
\end{corollary}
\begin{proof}
It remains to consider the case $k,\tilde{k}\in\overline{1,K-1}$ and $k\neq\tilde{k}$. Since then $(s_k,s_{\tilde{k}})_{\omega_h}=0$, the result follows from \eqref{sp1}.
\end{proof}
\par We call the calculation of $w\in S_K^{(n)}$ by the coefficients $w_{kl}$ of the expansion \eqref{eq:decomp} as \textit{the inverse $F_n$-transform} and the calculation of the coefficients $w_{kl}$ by $w\in S_K^{(n)}$ as \textit{the direct $F_n$-transform}.
We also consider related expansion of $y\in S_K^{(n)}$
\begin{gather}
 y=\sum_{l=1}^{n-1}\widetilde{y}_{0l}\mathcal{C}s_0^{(l)}+\sum_{k=1}^{K-1}\sum_{l=1}^n\widetilde{y}_{kl}\mathcal{C}s_k^{(l)}
\label{eq:decompC}
\end{gather}
and the calculation of the coefficients $\widetilde{y}_{kl}$ by $y\in S_K^{(n)}$ that we call as \textit{the direct $FC_n$-transform}.
\par Let us describe their fast FFT-based implementation.
\begin{theorem}
\label{th:eigpares 1}
1. The inverse $F_n$-transform can be implemented according to the following formulas
\begin{gather}
 w_j=\sum_{k=1}^{K-1}\Big(\sum_{l=1}^n w_{kl}\Big)\sin\frac{\pi kj}{K},\ \ j=\overline{1,K-1},
\label{eq:decomp_1}\\
 w_{j-1/2}=(-P)^{j-1}\sum_{l=1}^{n-1}w_{0l}e^{(l)}
\nonumber\\
 +2\sum_{k=1}^{K-1}d_{k,e}\cos\frac{\pi k}{2K}\sin\frac{\pi k(j-1/2)}{K}
 -2\sum_{k=1}^{K-1}d_{k,o}\sin\frac{\pi k}{2K}\cos\frac{\pi k(j-1/2)}{K},\ \ j=\overline{1,K},
\label{eq:decomp_2}
\end{gather}
where $d_{k,e}$ and $d_{k,o}$ are respectively even and odd components of the vector
$d_k:=\sum_{l=1}^n w_{kl}p_k^{(l)}$.
Note that $(-P)^{j-1}e=e$ for odd $j$ and $(-P)^{j-1}e=-\check{e}$ for even $j$ for any $e\in\mathbb{R}^{n-1}$.
\par The collection $\{w_j\}_{j=1}^{K-1}$ can be computed by the standard inverse FFT with respect to sines.
The collection  $\{w_{j-1/2}\}_{j=1}^K$ can be computed by $n-1$ modified inverse FFT related to  the centers of elements in the amount of $[n/2]$ with respect to sines and $[(n-1)/2]$  with respect to cosines using extensions $d_{K,e}:=0$ and $d_{0,o}:=0$, see algorithms DST-I, DST-III and DCT-III in \cite{BRY07}.
\smallskip\par 2. The direct $FC_n$-transform can be implemented basing on the standard formula
\begin{gather}
 \widetilde{y}_{kl}=\big(y,s_{k}^{(l)}\big)_{S_K^{(n)}}/\|s_k^{(l)}\|_{\mathcal{C}}^2.
\label{eq:coef decomp_2a}
\end{gather}
Here, first, for $k=0$, $l=\overline{1,n-1}$, the following formulas hold
\begin{gather}
  (y,s_0^{(l)})_{S_K^{(n)}}
 =\Big(\sum_{j=1}^K(-P)^{j-1}y_{j-1/2}\Big)\cdot e^{(l)}, \ \
 \|s_0^{(l)}\|_{\mathcal{C}}^2=K.
\label{eq:norms eigfa}
\end{gather}
\par Second, for $k=\overline{1,K-1}$, $l=\overline{1,n}$, the following formulas hold
\begin{gather}
 (y,s_k^{(l)})_{S_K^{(n)}}=\sum_{j=1}^{K-1}y_j\sin\frac{\pi kj}{K}
\nonumber\\
 +p_{k,e}^{(l)}\cdot\sum_{j=1}^{K-1}(y_{j-1/2}+y_{j+1/2})_e\sin\frac{\pi kj}{K}
 +p_{k,o}^{(l)}\cdot\sum_{j=1}^{K-1}(y_{j+1/2}-y_{j-1/2})_o\sin\frac{\pi kj}{K},
\label{eq:inv_fna}\\
 \|s_k^{(l)}\|_{\mathcal{C}}^2=K\Big\{c_0+\big(\widetilde{C}p_k^{(l)}+2c\big)\cdot p_k^{(l)}
 +\theta_k\big[c_n+\big(\widetilde{C}p_k^{(l)}+2c\big)\cdot \check{p}_k^{(l)}\big]\Big\}.
\label{eq:inv_fn2}
\end{gather}
Notice that the sums in formula \eqref{eq:inv_fna} are independent on $l$.
The collection of all these coefficients can be computed using $n$ standard direct FFTs  with respect to sines.
\smallskip\par 3. Similarly to Item 2, the direct $F_n$-transform can be implemented basing on the standard formula
\begin{gather}
 w_{kl}=\big(\mathcal{C}w,s_{k}^{(l)}\big)_{S_K^{(n)}}/\|s_k^{(l)}\|_{\mathcal{C}}^2.
\label{eq:coef decomp_2}
\end{gather}
Here, for $k=0$, $l=\overline{1,n-1}$, the following formula holds
\begin{gather}
  (\mathcal{C}w,s_0^{(l)})_{S_K^{(n)}}
 =\Big(\widetilde{C}\sum_{j=1}^K(-P)^{j-1}w_{j-1/2}\Big)\cdot e^{(l)}.
\label{eq:norms eigfC}
\end{gather}
\par For $k=\overline{1,K-1}$, $l=\overline{1,n}$, formula \eqref{eq:inv_fna} with $y:=\mathcal{C}w$ is applicable.
Alternatively, the following formula holds as well
\begin{gather}
 (\mathcal{C}w,s_k^{(l)})_{S_K^{(n)}}=2\big[c_0+c\cdot p_k^{(l)}
 +\theta_k\big(c_n+c\cdot\check{p}_k^{(l)}\big)\big]\sum_{j=1}^{K-1}w_j\sin\frac{\pi kj}{K}
\nonumber\\
 +q_{k,e}^{(l)}\cdot\sum_{j=1}^{K-1}(w_{j-1/2}+w_{j+1/2})_e\sin\frac{\pi kj}{K}
 +q_{k,o}^{(l)}\cdot\sum_{j=1}^{K-1}(w_{j+1/2}-w_{j-1/2})_o\sin\frac{\pi kj}{K},
\label{eq:inv_fnb}
\end{gather}
where $q_{k,e}^{(l)}$ and $q_{k,o}^{(l)}$ are respectively even and odd components of the vector $q_k^{(l)}:=\tilde{C}p_k^{(l)}+c$.
Once again all these coefficients can be computed using $n$ standard direct FFTs  with respect to sines.
\end{theorem}
\begin{proof}
1. Let the coefficients $w_{kl}$ of expansion \eqref{eq:decomp} be known.
According to the first formulas \eqref{eq:eig vec1} and \eqref{eq:eig vec2}, the values of $w$ for integer indices in \eqref{eq:decomp} are reduced to \eqref{eq:decomp_1}.
\par To compute $w$ for half-integer indices, we transform the second sum in \eqref{eq:decomp}.
Owing to decomposition \eqref{eq:eodecomp} we rewrite the second formula \eqref{eq:eig vec2} in the form
\begin{gather*}
  s_{k,j-1/2}^{(l)}=p_{k,e}^{(l)}\Big(\sin\frac{\pi k(j-1)}{K}+\sin\frac{\pi kj}{K}\Big)
                   -p_{k,o}^{(l)}\Big(\sin\frac{\pi kj}{K}-\sin\frac{\pi k(j-1)}{K}\Big)
\\[1mm]
  =2\cos\frac{\pi k}{2K}\,p_{k,e}^{(l)}\sin\frac{\pi k(j-1/2)}{K}
                   -2\sin\frac{\pi k}{2K}\,p_{k,o}^{(l)}\cos\frac{\pi k(j-1/2)}{K},\ \ j=\overline{1,K}.
\end{gather*}
Then using also the second formula \eqref{eq:eig vec1}, we obtain formula \eqref{eq:decomp_2}.
\par 2. Now we consider the computation of the coefficients in expansion \eqref{eq:decompC} for given $y\in S_K$.
Owing to the orthogonality property \eqref{eq:ort}, they first can be expressed in the form \eqref{eq:coef decomp_2a}
for $k=0$, $l=\overline{1,n-1}$ and $k=\overline{1,K-1}$, $l=\overline{1,n}$.
\par Formulas \eqref{sp2} and \eqref{eq:eig vec1} imply that
\begin{gather*}
 \|s_0^{(l)}\|_{\mathcal{C}}^2=K\widetilde{C}e^{(l)}\cdot e^{(l)}=K,\ \ l=\overline{1,n-1}.
\label{eq:norms eigf}
\end{gather*}
Lemma \ref{lem:scalar_prod1} immediately implies formula \eqref{eq:inv_fn2} since $(s_k,s_k)_{\omega_h}=K/2$.
\par By virtue of formulas \eqref{eq:eig vec1} for the numerator of formula \eqref{eq:coef decomp_2a} for $k=0$ we can write
\begin{gather*}
  (y,s_0^{(l)})_{S_K^{(n)}}
 =\sum_{j=1}^Ky_{j-1/2}\cdot (-P)^{j-1}e^{(l)}
 =\Big(\sum_{j=1}^K(-P)^{j-1}y_{j-1/2}\Big)\cdot e^{(l)}.
\end{gather*}
\par By virtue of formulas \eqref{eq:eig vec2} for the same numerator for $k=\overline{1,K-1}$ we get
\begin{gather}
 (y,s_k^{(l)})_{S_K^{(n)}}=\sum_{j=1}^{K-1}y_js_{k,j}
 +\sum_{j=1}^Ky_{j-1/2}s_{k,j-1}\cdot p_k^{(l)}
 +\sum_{j=1}^Ky_{j-1/2}s_{k,j}\cdot \check{p}_k^{(l)}.
\label{eq:inv_fna0}
\end{gather}
Therefore shifting by 1 the index in the second of these sums, and applying the identity
$a_1b_1+a_2b_2=0.5(a_1+a_2)(b_1+b_2)+0.5(a_1-a_2)(b_1-b_2)$ and recalling decomposition \eqref{eq:eodecomp}, we derive
\begin{gather*}
 (y,s_k^{(l)})_{S_K^{(n)}}
\\
 =\sum_{j=1}^{K-1}y_js_{k,j}
 +p_{k,e}^{(l)}\cdot\sum_{j=1}^{K-1}(y_{j-1/2}+y_{j+1/2})s_{k,j}
 +p_{k,o}^{(l)}\cdot\sum_{j=1}^{K-1}(y_{j+1/2}-y_{j-1/2})s_{k,j}.
\end{gather*}
Since also
\begin{gather*}
 p_e\cdot q =p_e\cdot q_e,\ \ p_o\cdot q =p_o\cdot q_o\ \ \text{for any}\ \ q,p\in \mathbb{R}^{n-1},
\label{eq:eodprop2}
\end{gather*}
we obtain formula \eqref{eq:inv_fna}.
\smallskip\par 3. Owing to the orthogonality property \eqref{eq:ort}, formula \eqref{eq:coef decomp_2} is valid.
\par By virtue of formulas \eqref{eq:eig vec1}, \eqref{sp2} as well as $\widetilde{C}^T=\widetilde{C}$ and $P=P^T$ for its numerator for $k=0$ we can write
\begin{gather*}
 \big(\mathcal{C}w,s_0^{(l)}\big)_{S_K^{(n)}}
 =\big(\mathcal{C}s_0^{(l)},w\big)_{S_K^{(n)}}
 =
 \sum_{j=1}^K\widetilde{C}(-P)^{j-1}e^{(l)}\cdot w_{j-1/2}
 =\Big(\widetilde{C}\sum_{j=1}^K(-P)^{j-1}w_{j-1/2}\Big)\cdot e^{(l)}.
\end{gather*}
\par By virtue of formulas \eqref{sp4}, \eqref{sp7} for the same numerator for $k=\overline{1,K-1}$ we get
\begin{gather*}
 \big(\mathcal{C}w,s_k^{(l)}\big)_{S_K^{(n)}}
 =\big(\mathcal{C}s_k^{(l)},w\big)_{S_K^{(n)}}
\\[1mm]
 =2\big[c_0+c\cdot p_k^{(l)}+\theta_k\big(c_n+c\cdot \check{p}_k^{(l)}\big)\big](s_k,w)_{\omega_h}
 +\sum_{j=1}^K\big(q_k^{(l)}s_{k,j-1}+\check{q}_k^{(l)}s_{k,j}\big)\cdot w_{j-1/2},
\end{gather*}
where $q_k^{(l)}=\tilde{C}p_k^{(l)}+c$.
Transforming the last sum in the same manner as above the second and third terms in \eqref{eq:inv_fna0}, we obtain \eqref{eq:inv_fnb}.
\end{proof}
\section{\large Applications to the generalized Poisson equation}
\label{sect5}
\setcounter{equation}{0}
\setcounter{proposition}{0}
\setcounter{theorem}{0}
\setcounter{lemma}{0}
\setcounter{corollary}{0}
\setcounter{remark}{0}
\par To begin with, we turn to the simple 1D ODE boundary value problem
\begin{gather}
 -u''(x)+\alpha u(x)=f(x)\ \ \text{on}\ \ [0,X],\ \ u(0)=u(X)=0,
\label{eq:diff bvp pr}
\end{gather}
where for simplicity $\alpha=\textrm{const}>-(\pi/X)^2$.
Its FEM discretization has the operator form
\begin{gather}
 4h^{-2}\mathcal{A}v+\alpha\mathcal{C}v=f^h,\ \ v\in S_K^{(n)},
\label{eq:bvp glob}
\end{gather}
where $f^h\in S_K^{(n)}$ is the FEM average of $f$.
Its solution can be written in the form
\begin{gather}
 v=\sum_{k=0}^K\,\sum_{l=1}^{n-\delta_{k0}}\frac{\widetilde{f}^h_{kl}}{4h^{-2}\lambda_k^{(l)}+\alpha}s_k^{(l)},
\label{eq:sol decomp 1}
\end{gather}
of the expansion like \eqref{eq:decomp}, where
$\widetilde{f}^h_{kl}$
are the coefficients of the expansion like \eqref{eq:decompC} for the vector $f^h$;
recall that $\delta_{k0}$ is the Kronecker delta.
\par Next we consider in detail solving of the $N$-dimensional ($N\geq 2$) boundary value problem
\begin{gather}
 -\Delta u+\alpha u=f\ \ \text{в}\ \ \Omega=(0,X_1)\times\ldots\times(0,X_N),\ \ u|_{\partial\Omega}=0,
\label{eq:diff bvp pr2}
\end{gather}
where $\Delta$ is the Laplace operator and $\alpha=\textrm{const}>-\pi^2\big(X_1^{-2}+\ldots+X_N^{-2}\big)$ (for simplicity, in order to treat the positive definite operator that actually is not so necessary).
\par We define the space $H_{h_1}^{(n_1)}[0,X_1]\otimes\ldots\otimes H_{h_N}^{(n_N)}[0,X_N]$ of the piecewise-poly\-no\-mial in
$\overline{\Omega}$ functions, where  $h_i=X_i/K_i$ and $n_i\geq 2$, $i=\overline{1,N}$.
Let $\mathbf{K}=(K_1,\ldots,K_N)$ and $\mathbf{n}=(n_1,\ldots,n_N)$.
\par We also define the space $S_{\mathbf{K}}^{(\mathbf{n})}=S_{K_1}^{(n_1)}\otimes\ldots\otimes S_{K_N}^{(n_N)}$ of vector functions.
For example, for $N=2$, these functions
are numbers for the indices $(j_1,j_2)$, $j_1=\overline{0,K_1}$, $j_2=\overline{0,K_2}$, and
vectors from $\mathbb{R}^{n_1-1}$, $\mathbb{R}^{n_2-1}$ and $\mathbb{R}^{(n_1-1)\times (n_2-1)}$ respectively for the indices
\begin{gather*}
(j_1-1/2,j_2), j_1=\overline{1,K_1}, j_2=\overline{0,K_2};\ \
(j_1,j_2-1/2), j_1=\overline{0,K_1}, j_2=\overline{1,K_2}\ \ \text{and}\ \
\\
(j_1-1/2,j_2-1/2), j_1=\overline{1,K_1}, j_2=\overline{1,K_2},
\end{gather*}
as well as zero vectors for $j_1=0,K_1$ and $j_2=0,K_2$.
Similarly to the 1D case, there is the natural isomorphism between
functions in $H_{h_1}^{(n_1)}[0,X_1]\otimes\ldots\otimes H_{h_N}^{(n_N)}[0,X_N]$ and vectors in $S_{\mathbf{K}}^{(\mathbf{n})}$.
\par The FEM dicretization of problem \eqref{eq:diff bvp pr2} can be written in the following operator form
\begin{gather}
 \big(4h_1^{-2}\mathcal{A}_1\mathcal{C}_2\ldots\mathcal{C}_N+\ldots+4h_N^{-2}\mathcal{A}_N\mathcal{C}_1\ldots\mathcal{C}_{m-1}\big)v
 +\alpha \mathcal{C}_1\ldots\mathcal{C}_Nv=f^h,\ \ v\in S_{\mathbf{K}}^{(\mathbf{n})},
\label{eq:bvp glob 2}
\end{gather}
where $\mathcal{A}_i$ and $\mathcal{C}_i$ are versions of the above defined operators $\mathcal{A}$ and $\mathcal{C}$ acting in variable  $x_i$ (depending on $K_i$ and $n_i$), $i=\overline{1,N}$, and $f^h\in S_{\mathbf{K}}^{(\mathbf{n})}$ is the FEM average of $f$.
Recall that the case of the non-homogeneous Dirichlet boundary condition $u(x)=b(x)$ on $\partial\Omega$ in \eqref{eq:diff bvp pr2} could be easily covered by reducing to \eqref{eq:bvp glob 2} with the modified $f^h$ depending on an approximation $b^h$ of $b$.
\par To compute its solution, the $F_n$-transforms from Theorem \ref{th:eigpares 1} can be applied twofold.
\smallskip\par (a)
We consider the multiple expansion of $f^h\in S_{\mathbf{K}}^{(\mathbf{n})}$ like \eqref{eq:decompC}
\begin{gather}
 f^h
 =\sum_{i=1}^N\sum_{k_i=0}^{K_i-1}\,\sum_{l_i=1}^{n_i-\delta_{k_i0}}
 \widetilde{f}^h_{k_1l_1,\ldots,k_Nl_N}\mathcal{C}_1s_{1,\,k_1}^{(l_1)}\ldots\mathcal{C}_Ns_{N,\,k_N}^{(l_N)}.
\label{eq:decomp phi 2}
\end{gather}
Then the expansion of the solution has the following form
\begin{gather}
 v=\sum_{i=1}^N\sum_{k_i=0}^{K_i-1}\,\sum_{l_i=1}^{n_i-\delta_{k_i0}}
 \frac{\widetilde{f}^h_{k_1l_1,\ldots,k_Nl_N}}
 {4h_1^{-2}\lambda_{1,\,k_1}^{(l_1)}+\ldots+4h_N^{-2}\lambda_{N,\,k_N}^{(l_N)}+\alpha}
 s_{1,\,k_1}^{(l_1)}\ldots s_{N,\,k_N}^{(l_N)}.
\label{eq:decomp v 2}
\end{gather}
Here $\big\{\lambda_{i,k_i}^{(l_i)},s_{i,k_i}^{(l_i)}\big\}$ are versions of the above defined eigenpairs
$\big\{\lambda_{k}^{(l)},s_{k}^{(l)}\big\}$ with respect to $x_i$.
\par Algorithm (a) comprises two rather standard steps:
\par (1) finding the coefficients of expansion \eqref{eq:decomp phi 2} for $f^h$ (by the direct $FC_n$-transforms in $x_1$,..., $x_N$);
\par (2) finding $v$ by the coefficients of its expansion \eqref{eq:decomp v 2} (by the inverse $F_n$-transforms in $x_1$,..., $x_N$).

\smallskip\par (b) We consider the expansion of $f^h$ like \eqref{eq:decompC} in $x_2$,..., $x_N$, i.e.
\begin{gather}
 f^h=
 \sum_{i=2}^N\sum_{k_i=0}^{K_i-1}\,\sum_{l_i=1}^{n_i-\delta_{k_i0}}\widetilde{f}^h_{k_2l_2,\ldots,k_Nl_N}
 \mathcal{C}_2s_{2,\,k_2}^{(l_2)}\ldots \mathcal{C}_Ns_{N,\,k_N}^{(l_N)},
\label{eq:decomp phi}
\end{gather}
now with the coefficients $\widetilde{f}^h_{k_2l_2,\ldots,k_Nl_N}\in S_{K_1}^{(n_1)}$.
Then the coefficients $v_{kl}\in S_{K_1}^{(n_1)}$ in the similar expansion of the solution $v\in S_{\mathbf{K}}^{(\mathbf{n})}$
\begin{gather}
 v=
 \sum_{i=2}^N\sum_{k_i=0}^{K_i-1}\,\sum_{l_i=1}^{n_i-\delta_{k_i0}}
 v_{k_2l_2,\ldots,k_Nl_N}s_{2,\,k_2}^{(l_2)}\ldots s_{m,\,k_N}^{(l_N)},
\label{eq:decomp v}
\end{gather}
serve as the solutions to 1D problems in $x_1$
\begin{gather}
 \big[4h_1^{-2}\mathcal{A}_1
 +\big(4h_2^{-2}\lambda_{k_2}^{(l_2)}+\ldots+4h_N^{-2}\lambda_{k_N}^{(l_N)}+\alpha\big)\mathcal{C}_1\big]v_{k_2l_2,\ldots,k_Nl_N}
 =\widetilde{f}^h_{k_2l_2,\ldots,k_Nl_N}.
\label{eq:sys x1 2}
\end{gather}
Their matrices are symmetric and positive definite.
\par Algorithm (b) comprises three rather standard steps:
\par (1)  finding the coefficients of the expansion \eqref{eq:decomp phi} for $f^h$ (by the direct $FC_n$-transforms in $x_2$,..., $x_N$);
\par (2) solving the collection of the independent 1D problems \eqref{eq:sys x1 2} for the coefficients of the expansion of $v$;
\par (3) finding $v$ by the coefficients of its expansion  \eqref{eq:decomp v} (by the inverse $F_n$-transforms in $x_2$,..., $x_N$).

\par Implementing algorithms (a) and (b) costs respectively $O\big(K_1\ldots K_N\log_2(K_1\ldots K_N)\big)$ and
$O\big(K_1\ldots K_N\log_2(K_2\ldots K_N)\big)$ arithmetic operations.
\par Importantly, they can be applied to solve various time-dependent PDEs such as the heat, wave or Schr\"{o}dinger's equations since for their implicit time discretizations one usually gets problems like \eqref{eq:bvp glob 2} at the upper time level.

\par Moreover, algorithm (b) is directly extended to the case of more general equations than in \eqref{eq:diff bvp pr2} with the coefficients depending on $x_1$ (that is essential, in particular, in the polar and cylindrical coordinates), various boundary conditions for $x_1=0,X_1$ and the nonuniform mesh in $x_1$ \cite{SN78}.
It can also be applied for reducing 3D problems in a cylindrical domain to a collection of independent 2D problems in the cylinder base.

\section{\large Numerical experiments}
\label{sect6}
\setcounter{equation}{0}
\setcounter{proposition}{0}
\setcounter{theorem}{0}
\setcounter{lemma}{0}
\setcounter{corollary}{0}
\setcounter{remark}{0}

\par 1. We first check that the eigenvalues of each of problems \eqref{eq:eig_elem} and \eqref{eq:eig_elemt} are well separated.
We define their spectral gaps as
\[
 \min_{1\leq l\leq n}(\lambda_{0(l+1)}-\lambda_{0(l)})=\frac{\pi^2}{4}+\delta_n,\ \
 \min_{1\leq l\leq n-2}(\lambda_0^{(l+1)}-\lambda_0^{(l)})=\frac{3\pi^2}{4}+\tilde{\delta}_n,
\]
where $S_n=:\{\lambda_{0(l)}\}_{l=1}^{n+1}$ and present $\tilde{\delta}_n$ and $\tilde{\delta}_n$ in Fig. \ref{fig:con num and spectral gap} (left).
The terms $\pi^2/4$ and $3\pi^2/4$ are the spectral gaps (in fact, the gaps between two minimal eigenvalues) of the corresponding ODE problems, see \cite{ZZ12}.
We observe that both $\delta_n$ and $\tilde{\delta}_n$ are decreasing and rapidly tend to 0 as $n$ increases.
We also checked that $S_n\cap\tilde{S}_n=\emptyset$ for all $2\leq n\leq 21$.
\par Also we give the spectral condition numbers $\cond\widetilde{A}$ and $\cond\widetilde{C}$ in
Fig. \ref{fig:con num and spectral gap} (right) and remark their rapid growth as $n$ increases (unfortunately).
\par Notice that all our computations are accomplished on an ordinary notebook ASUS-U36S with Intel Core i3-2350M CPU 2.3 GHz, 8 Gb, Win 10~x64. The codes in Matlab R2016a were developed to implement the algorithms, and
we emphasize that several basic and advanced \cite{A14} code vectorisation techniques were applied to speed up them notably.
\begin{figure}[htbp]\centering{
    \begin{minipage}[h]{0.49\linewidth}
        \center{\includegraphics[width=1\linewidth]{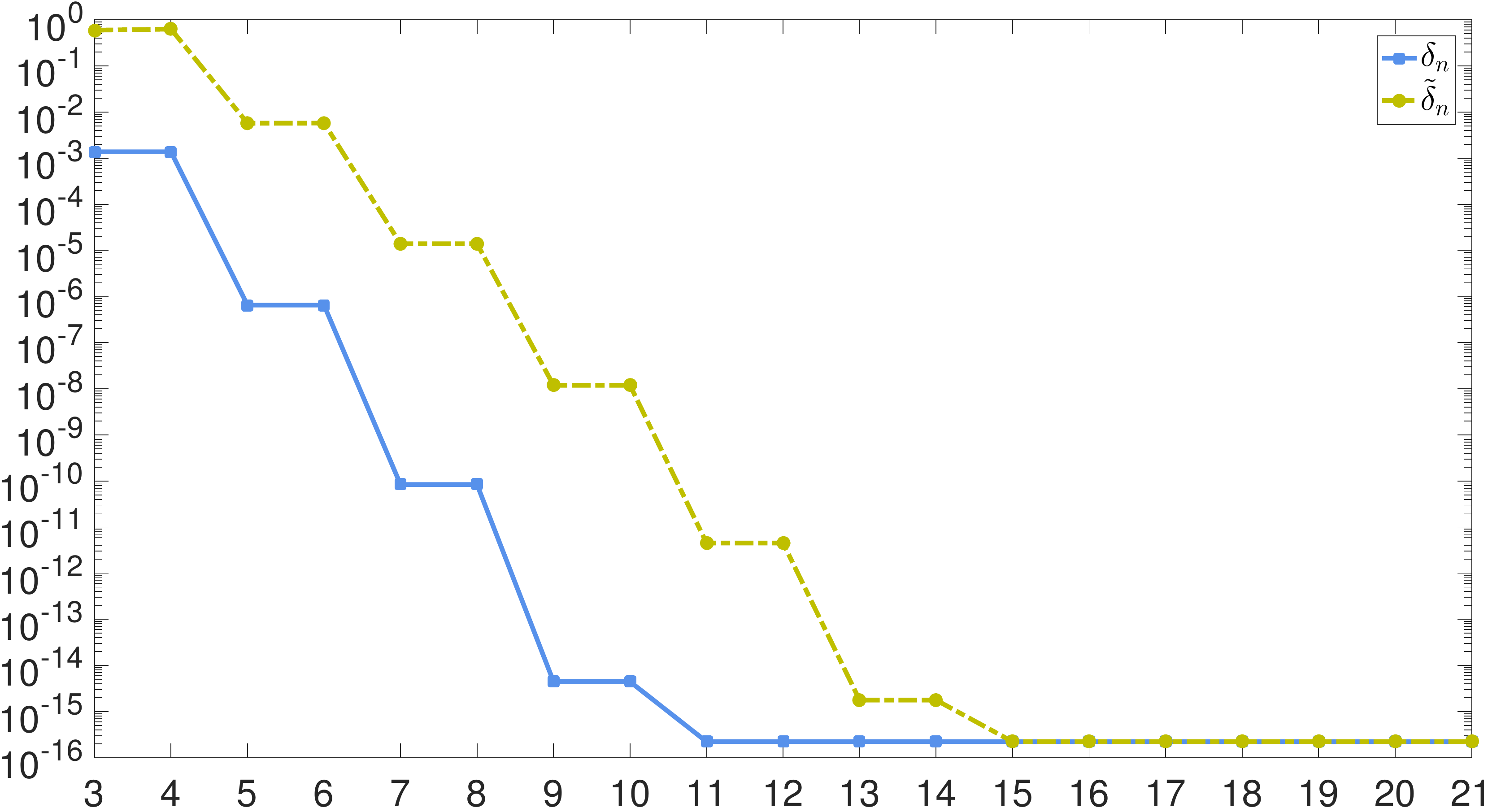}}
    \end{minipage}
    \begin{minipage}[h]{0.49\linewidth}
        \center{\includegraphics[width=1\linewidth]{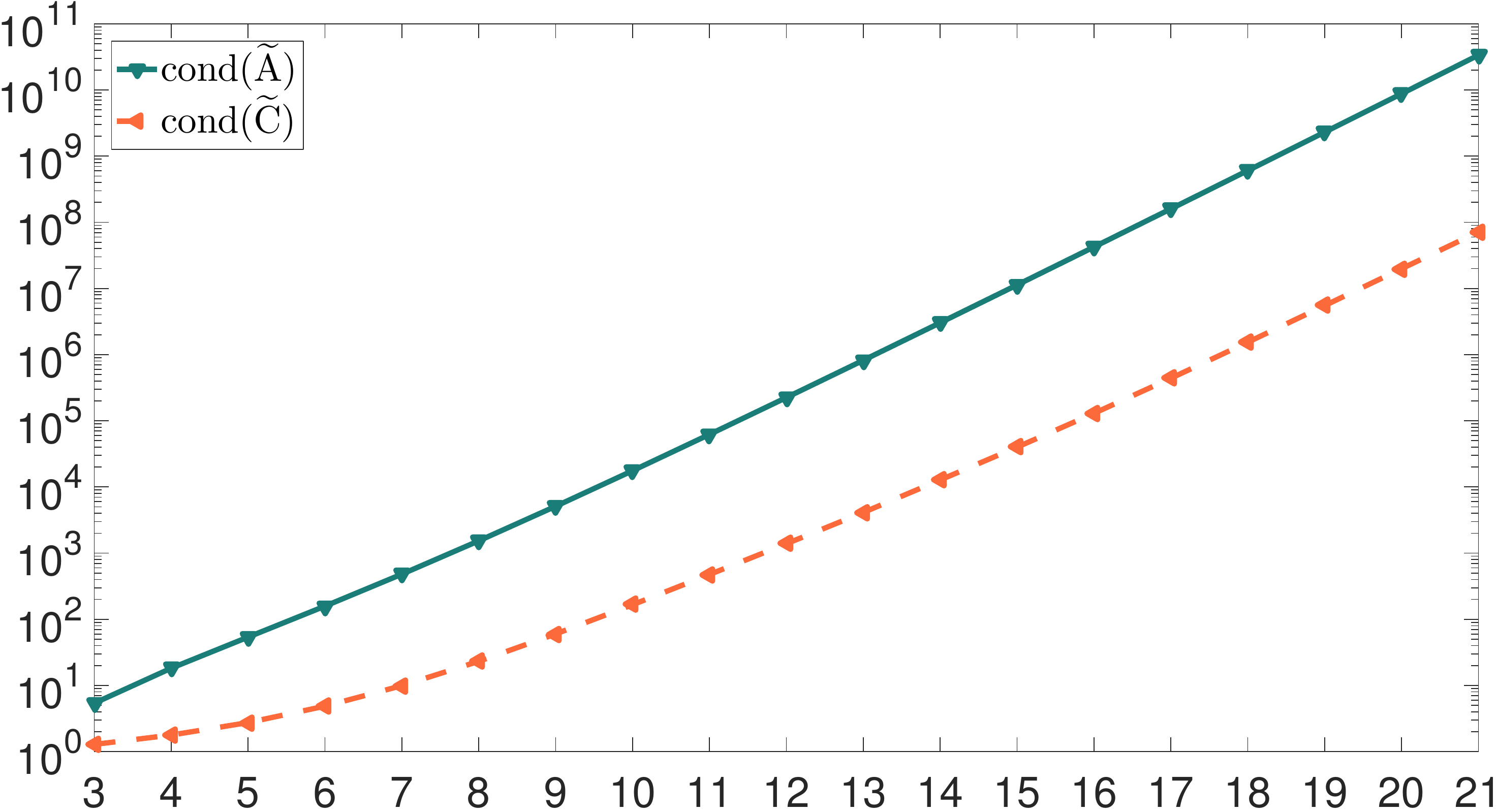}}
    \end{minipage}
    }
\caption{\small{The numbers $\delta_n$ and $\tilde{\delta}_n$ related to the minimal spectral gaps of eigenvalue problems \eqref{eq:eig_elem} and \eqref{eq:eig_elemt} (left) as well as  $\cond\widetilde{A}$ and $\cond\widetilde{C}$ (right) for problem \eqref{eq:eig_elemt} in dependence on $n$}}
\label{fig:con num and spectral gap}
\end{figure}

In the case of the 1D ODE problem \eqref{eq:bvp glob} for $\alpha=1$, we provide the condition numbers of the global FEM matrix $4h^{-2}\mathcal{A}+\alpha \mathcal{C}$ for $\alpha=1$
and its local version in dependence on $K=2,4,...,1024$ for $n=\overline{1,9}$ in Fig. \ref{fig:Cond:K} for a further reference.
Notice their rather rapid growth as $K$ increases.
\begin{figure}[htbp]\centering{
    \begin{minipage}[h]{0.49\linewidth}
        \center{\includegraphics[width=1\linewidth]{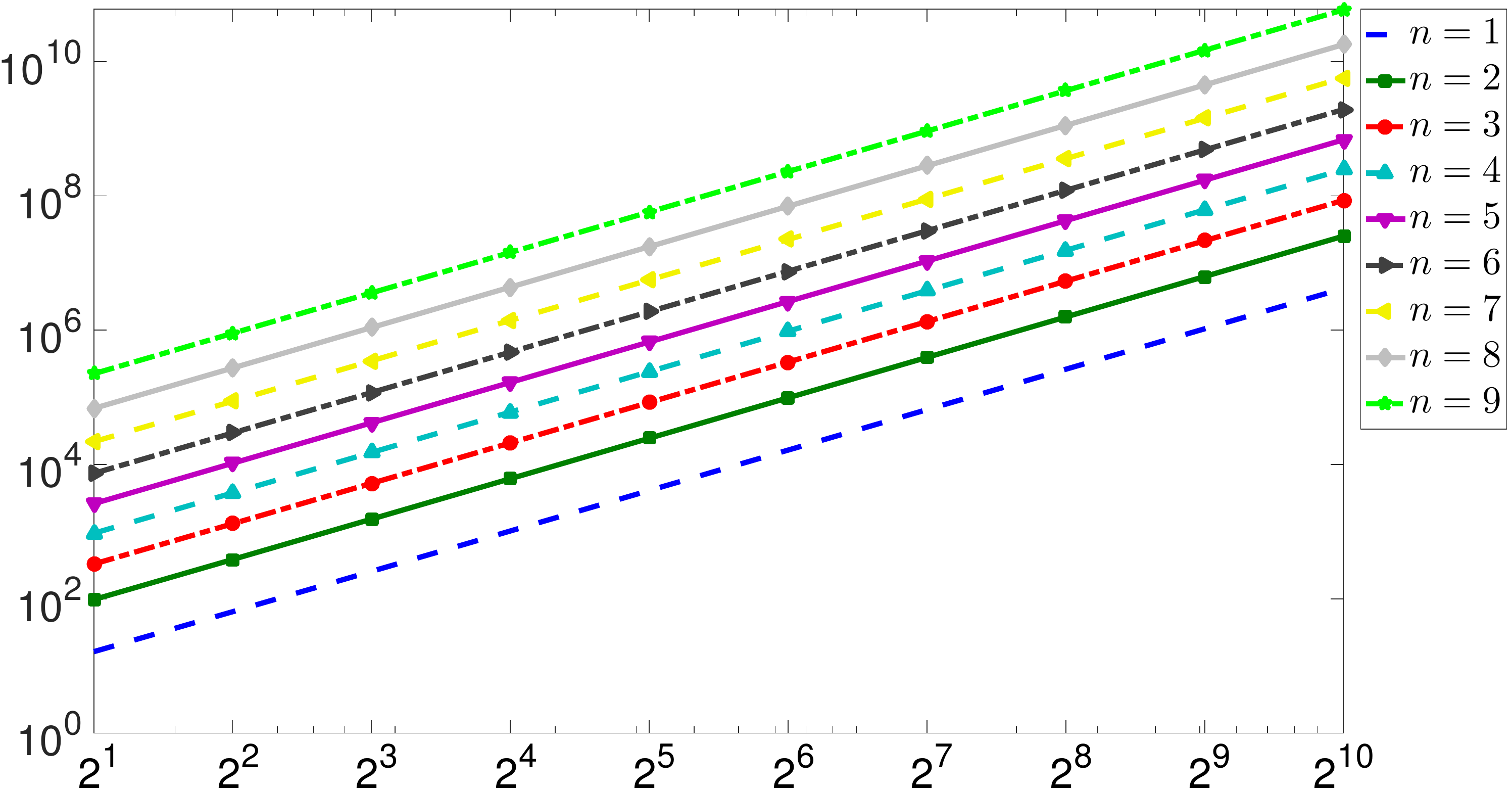}}
    \end{minipage}
    \begin{minipage}[h]{0.49\linewidth}
        \center{\includegraphics[width=1\linewidth]{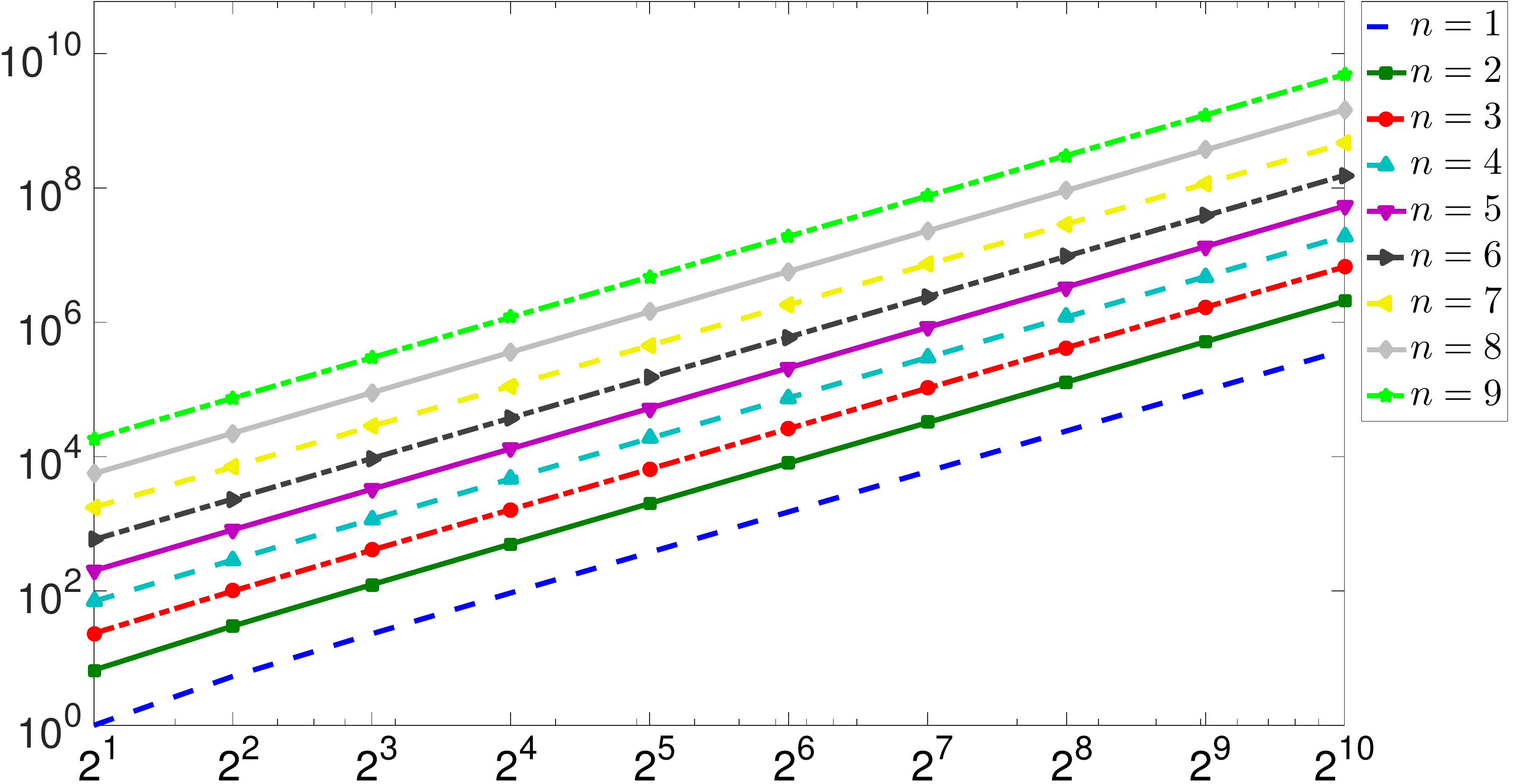}}
    \end{minipage}
    }\caption{\small{The condition numbers of the local (left) and global (right) matrices in 1D problem \eqref{eq:bvp glob} for $\alpha=1$ in dependence on $K=2,4,...,1024$ for $n=\overline{1,9}$}}
\label{fig:Cond:K}
\end{figure}
\par Next we analyze the execution time for $F_n$-transforms in dependence on the choice of $K$ as it grows up to very high values $2^{20}$.
We consider three choices of $K$: powers of two ($K=2^p$), primes or other composite
numbers (not powers of two).
Primes and composite values of $K$ are randomly chosen as different ten numbers between each two consecutive powers of two, and
the comparison is accomplished in the spirit of \cite{PrimeAnalysis}.
We use the external function {\tt comp\_dst} for DST-I and DST-III  and {\tt comp\_dct} for DCT-III from Large Time-Frequency Analysis Toolbox \cite{LTFAT} based on FFTW library \cite{FJ05}, which is also used in the {\bf Matlab} function {\tt fft} for FFT.
\par  The execution times for the DST-I and DST-III as well as our transforms in the case $n=5$ (for definiteness) are given respectively in Fig. \ref{fig:TIME:DST:ASUS}  and Fig. \ref{fig:TIME:FEMFFT:ASUS}.
We do not take into account the execution time for the pairs $\{\lambda^{(l)}_k,p_k^{(l)}\}$ since it becomes negligible for multiple using of the transforms for fixed $K$ as required below.
For our transforms the execution times are respectively larger (than for the DST-I and DST-III) due to multiple using of FFTs and some additional computations.
The inverse $F_n$-transform takes less time than the direct one.
Moreover, the best results are mainly for $K=2^p$ though this is not the case for DST-I (fortunately, we apply it in the optimal case $K=2^p-1$).
For the two other choices of $K$ the execution times are worse but close, and the difference between all of them is less than in the case of both DST-I and DST-III. These results are attractive.
Notice that the data in Fig. \ref{fig:TIME:FEMFFT:ASUS} can be approximated by linear functions for $2^5\leq K\leq 2^{10}$ and $2^{10}\leq K\leq 2^{20}$ but with rather flat slope in the former case and a visibly higher slope in the latter one.
\begin{remark}
\label{archit}
The last phenomenon is explained by the advanced architecture of modern processors involving cache memory, streaming SIMD extensions and advanced vector extensions of the instruction set, etc.
Also the above mentioned high-quality implementations of FFTs are applied.
\end{remark}
\begin{figure}[htbp]\centering{
    \begin{minipage}[h]{0.49\linewidth}             
        \center{\includegraphics[width=1\linewidth]{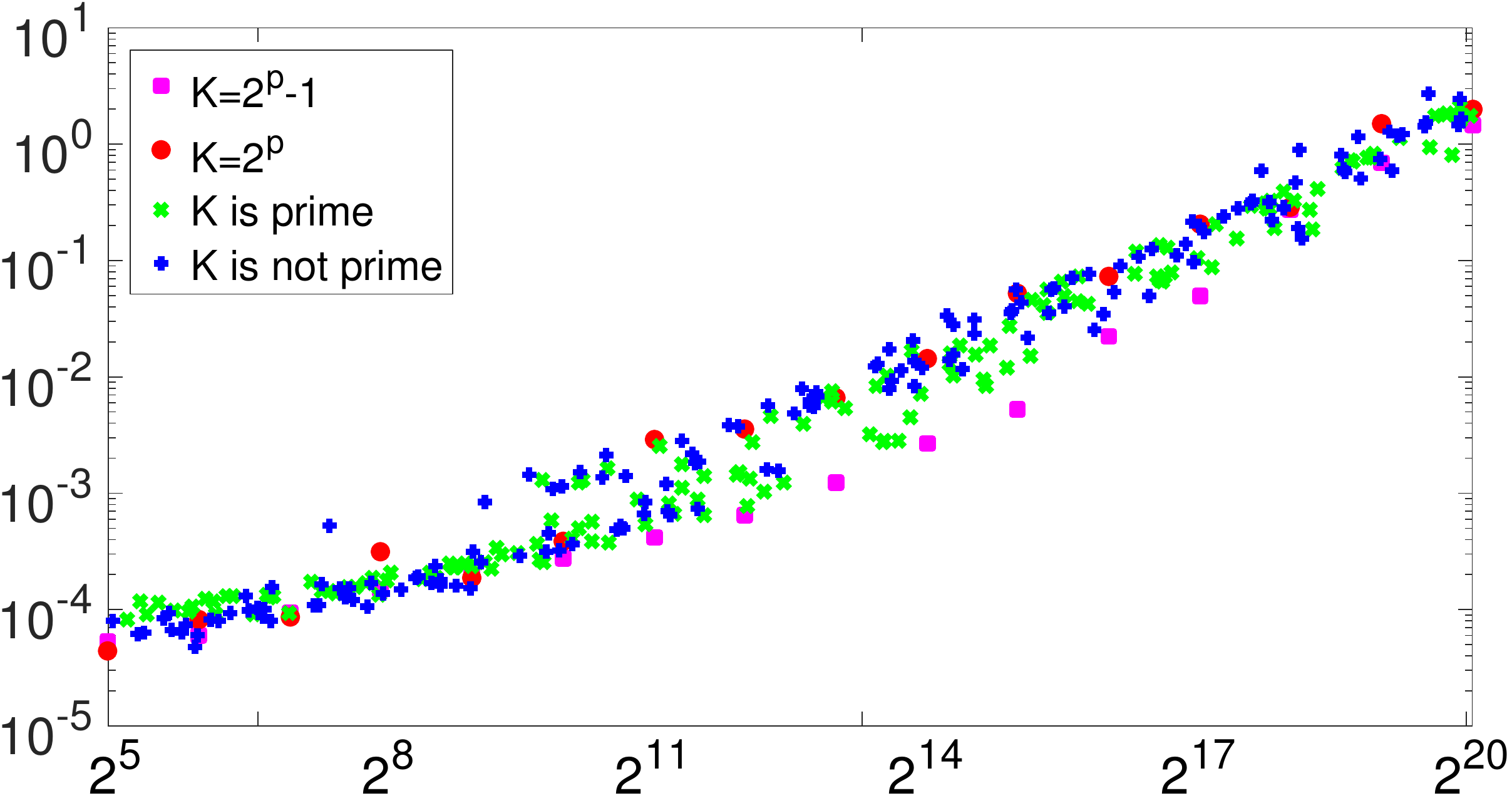}}
    \end{minipage}
    \begin{minipage}[h]{0.49\linewidth}
        \center{\includegraphics[width=1\linewidth]{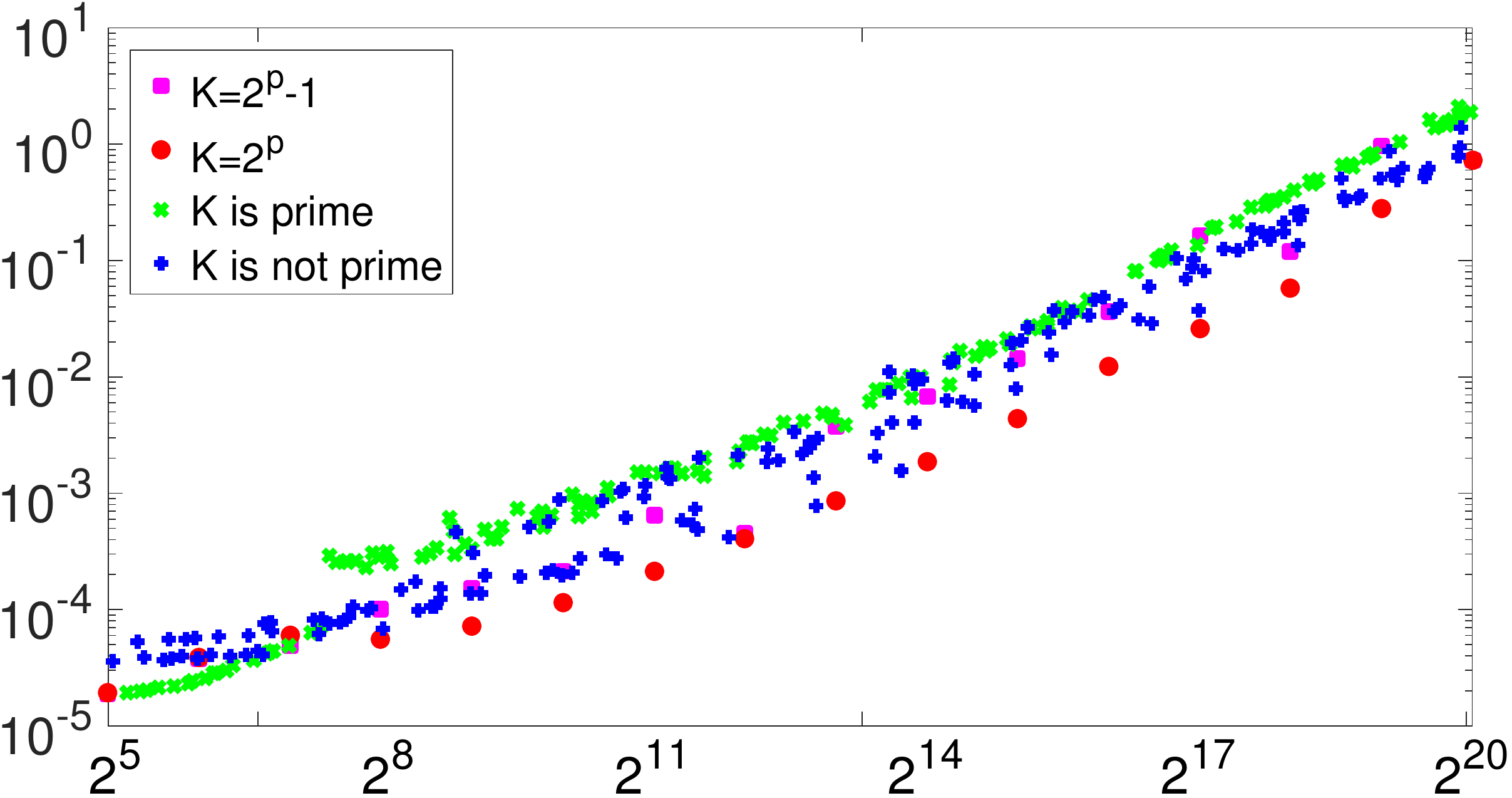}}
    \end{minipage}
    }\caption{\small{The execution time (in seconds) for DST-I (left) and DST-III (right) for the several choices of $K$ in between $2^5$ and $2^{20}$}}
\label{fig:TIME:DST:ASUS}
\end{figure}
\begin{figure}[htbp]\centering{
    \begin{minipage}[h]{0.49\linewidth}
        \center{\includegraphics[width=1\linewidth]{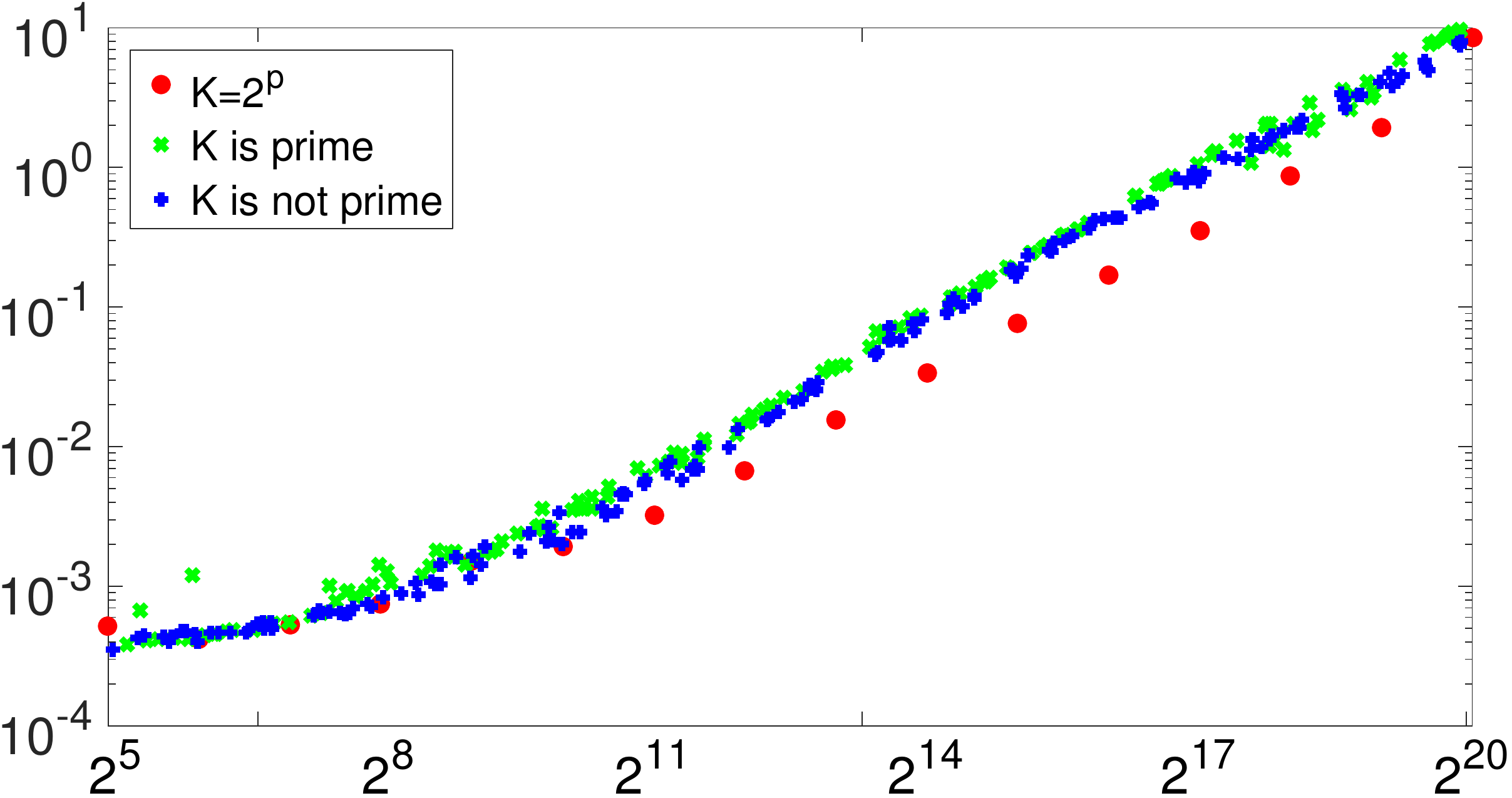}}
    \end{minipage}
    \begin{minipage}[h]{0.49\linewidth}
        \center{\includegraphics[width=1\linewidth]{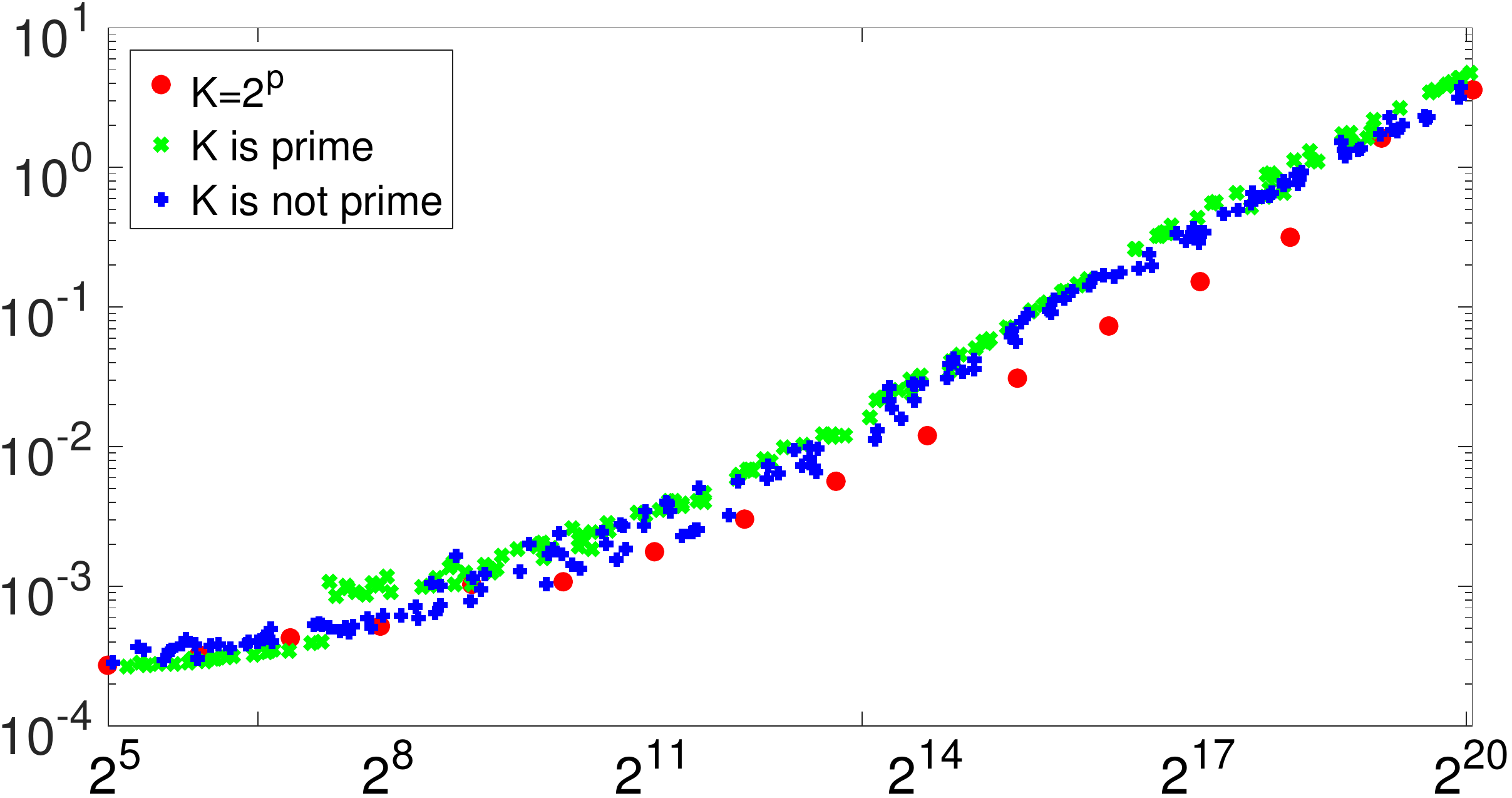}}
    \end{minipage}
    }\caption{\small{The execution time (in seconds) for the direct (left) and inverse (right) $F_n$-transforms, $n=5$, for the several choices of $K$ in between $2^5$ and $2^{20}$}}
\label{fig:TIME:FEMFFT:ASUS}
\end{figure}
\smallskip\par 2. Our main computational results concern solving problem \eqref{eq:diff bvp pr2} for both $N=2$ and 3, $\alpha=1$ and $X_i=1$.
We take $K_i=K$ and $n_i=n$ for simplicity.
We apply the multiple Gauss quadrature formulas with $n+1$ nodes in $x_i$ to compute $f^h$.
The eigenpairs of the 1D problems are computed with the quadruple precision (using Mathematica) to improve the stability with respect to round-off errors.
We notice that system \eqref{eq:bvp glob 2} contains $(Kn-1)^N$ unknowns.
For comparison, we include the simplest known case $n=1$ implemented in the code in the unified manner.

\par We first consider the 2D case ($N=2$) and take the exact solution
$u(x):=\sin(2\pi x_1)\sin(3\pi x_2)$ $\cosh\big(\sqrt{2}x_1-x_2\big)$.
We compute the FEM solution for different values of $n$ and $K$ in order to study the (absolute) error behaviour, see Fig. \ref{fig:EX72:2D:ErrorC} and Table \ref{tab:EX72:2D:FFT:ErrorCabs} (where values of $R_C$ less than 3 are omitted).
To reduce the possible round-off errors, hereafter we compute the pairs $\{\lambda^{(l)}_k,p_k^{(l)}\}$ with the quadruple precision.
For algorithm (a), the error behaviour in the uniform mesh norm is standard: the rate of its decreasing $R_C$ is mainly proportional to $(n+1)^2$ except for $n=2$ when it is faster and similar to $n=3$ (the exception has previously been noted in \cite{DFN09}).
Curiously, for $n=9$ and the minimal $K=2$ the error is already less than for $n=1$ and the maximal $K=1024$.
Of course, the error cannot be better than the level of round-off errors that is achieved the faster, the higher $5\leq n\leq 9$.
We observe that impact of the round-off errors is almost absent.
For algorithm (b), the situation is similar only up to the error level $\sim$$10^{-11}$.
Once this value is reached for fixed $3\leq n\leq 9$, we further see the error growth as $K$ increases that means the perceptible
impact of the round-off errors.
This is due to the respective growth of condition numbers for matrices in system \eqref{eq:sys x1 2} as $K$ or $n$ increases, see above Fig. \ref{fig:Cond:K}.
Consequently algorithm (a) is preferable than (b) provided that very high precision is required.
Notice that if we use the double precision olny, then the level of the best error becomes at $\sim$$10^{-12}-10^{-13}$ and the results remain stable for algorithm (a) but they remain practically unchanged for algorithm (b).
Thus only the double precision computations are possible provided that the mentioned accuracy is sufficient (that is the case in a lot of applications).
\par We also analyze the execution time for both algorithms (using multiple program runs
and their median execution time), see Fig. \ref{fig:TIME:2D:Tsol} and Tables \ref{tab:TIME:2D:FFT:Tsol} and \ref{tab:TIME:2D:CHOL:Tsol} for the same $K$ and $n$. Clearly this time is independent of the above specific choice of $u$.
We do not take into account the execution time for the pairs $\{\lambda^{(l)}_k,p_k^{(l)}\}$ considering the case where they are computed in advance and stored (recall that they are independent of the data of PDE problem \eqref{eq:diff bvp pr2}).
We call attention to the weakly superlinear behaviour of time in $K$ and its mild monotone growth in $n$.
Notice that all the ratios of the consecutive execution times in the both tables are even less than the lower bound 4 for the theoretical ratios (the ratios less than one are omitted); see Remark \ref{archit} in this respect.
In contrast to theoretical expectations, almost all ratios for algorithm (a) are less than for (b).
The ratios grows as $K$ and $n$ increase, and for algorithm (b) the highest value is already close to 4.
For the maximal $K=1024$ and $n=9$, system \eqref{eq:bvp glob 2} contains almost $85\cdot10^6$ unknowns but only less than 2 min is required to solve it that is the excellent result.
\begin{figure}[htbp]\centering{
    \begin{minipage}[h]{0.49\linewidth}
        \center{\includegraphics[width=1\linewidth]{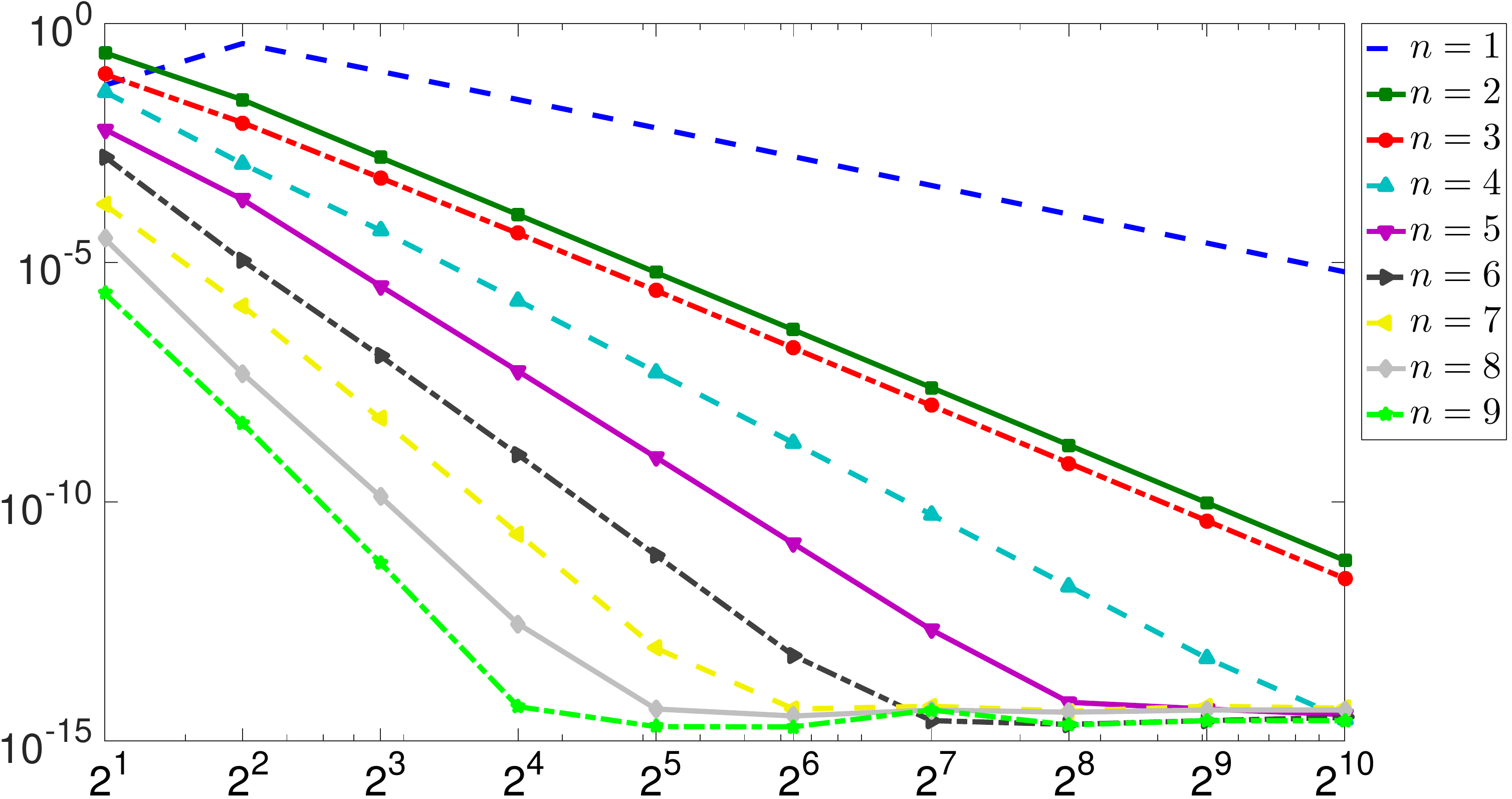}}\\ (a)
    \end{minipage}
    \begin{minipage}[h]{0.49\linewidth}
        \center{\includegraphics[width=1\linewidth]{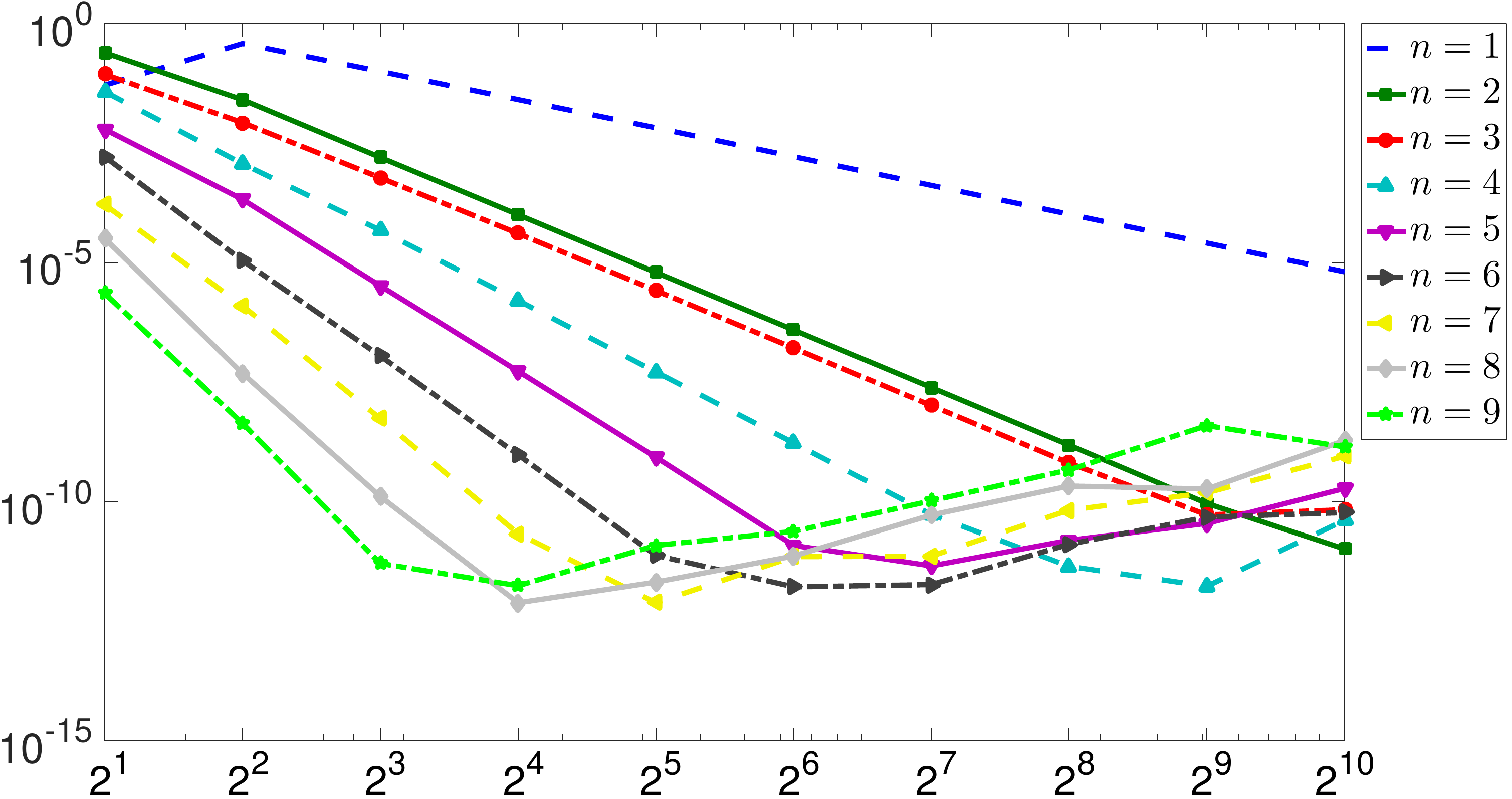}}\\ (b)
    \end{minipage}
    }\caption{\small{The 2D case. The errors in the mesh uniform norm for algorithms (a) and (b) in dependence on $K=2,4,...,1024$ for $n=\overline{1,9}$}}
\label{fig:EX72:2D:ErrorC}
\end{figure}
\begin{table}\centering{
\begin {tabular}{r<{\pgfplotstableresetcolortbloverhangright }@{}l<{\pgfplotstableresetcolortbloverhangleft }r<{\pgfplotstableresetcolortbloverhangright }@{}l<{\pgfplotstableresetcolortbloverhangleft }r<{\pgfplotstableresetcolortbloverhangright }@{}l<{\pgfplotstableresetcolortbloverhangleft }r<{\pgfplotstableresetcolortbloverhangright }@{}l<{\pgfplotstableresetcolortbloverhangleft }r<{\pgfplotstableresetcolortbloverhangright }@{}l<{\pgfplotstableresetcolortbloverhangleft }r<{\pgfplotstableresetcolortbloverhangright }@{}l<{\pgfplotstableresetcolortbloverhangleft }r<{\pgfplotstableresetcolortbloverhangright }@{}l<{\pgfplotstableresetcolortbloverhangleft }r<{\pgfplotstableresetcolortbloverhangright }@{}l<{\pgfplotstableresetcolortbloverhangleft }r<{\pgfplotstableresetcolortbloverhangright }@{}l<{\pgfplotstableresetcolortbloverhangleft }r<{\pgfplotstableresetcolortbloverhangright }@{}l<{\pgfplotstableresetcolortbloverhangleft }r<{\pgfplotstableresetcolortbloverhangright }@{}l<{\pgfplotstableresetcolortbloverhangleft }}%
\toprule \multicolumn {2}{c}{$K$}&\multicolumn {2}{c}{$n=1$}&\multicolumn {2}{c}{$R_{C}$}&\multicolumn {2}{c}{$n=2$}&\multicolumn {2}{c}{$R_{C}$}&\multicolumn {2}{c}{$n=3$}&\multicolumn {2}{c}{$R_{C}$}&\multicolumn {2}{c}{$n=4$}&\multicolumn {2}{c}{$R_{C}$}&\multicolumn {2}{c}{$n=5$}&\multicolumn {2}{c}{$R_{C}$}\\\midrule %
$2$&$$&$5$&$.1\cdot 10^{-2}$&--&&$2$&$.4\cdot 10^{-1}$&--&&$8$&$.7\cdot 10^{-2}$&--&&$3$&$.7\cdot 10^{-2}$&--&&$6$&$.1\cdot 10^{-3}$&--&\\%
$4$&$$&$3$&$.8\cdot 10^{-1}$&--&&$2$&$.5\cdot 10^{-2}$&$9$&$.6$&$8$&$.4\cdot 10^{-3}$&$10$&$.3$&$1$&$.2\cdot 10^{-3}$&$32$&$.2$&$2$&$.1\cdot 10^{-4}$&$29$&$.3$\\%
$8$&$$&$1$&$.0\cdot 10^{-1}$&$3$&$.8$&$1$&$.6\cdot 10^{-3}$&$16$&$.1$&$5$&$.9\cdot 10^{-4}$&$14$&$.2$&$4$&$.7\cdot 10^{-5}$&$24$&$.5$&$3$&$.3\cdot 10^{-6}$&$63$&$.8$\\%
$16$&$$&$2$&$.6\cdot 10^{-2}$&$3$&$.9$&$1$&$.0\cdot 10^{-4}$&$15$&$.7$&$4$&$.1\cdot 10^{-5}$&$14$&$.6$&$1$&$.6\cdot 10^{-6}$&$29$&$.3$&$5$&$.4\cdot 10^{-8}$&$60$&$.8$\\%
$32$&$$&$6$&$.6\cdot 10^{-3}$&$3$&$.9$&$6$&$.2\cdot 10^{-6}$&$16$&$.1$&$2$&$.6\cdot 10^{-6}$&$15$&$.6$&$5$&$.2\cdot 10^{-8}$&$31$&$.1$&$8$&$.5\cdot 10^{-10}$&$63$&$.2$\\%
$64$&$$&$1$&$.6\cdot 10^{-3}$&$4$&$.0$&$3$&$.9\cdot 10^{-7}$&$15$&$.9$&$1$&$.6\cdot 10^{-7}$&$15$&$.9$&$1$&$.7\cdot 10^{-9}$&$30$&$.6$&$1$&$.3\cdot 10^{-11}$&$64$&$.0$\\%
$128$&$$&$4$&$.1\cdot 10^{-4}$&$4$&$.0$&$2$&$.4\cdot 10^{-8}$&$16$&$.0$&$1$&$.0\cdot 10^{-8}$&$16$&$.0$&$5$&$.4\cdot 10^{-11}$&$31$&$.4$&$2$&$.1\cdot 10^{-13}$&$63$&$.1$\\%
$256$&$$&$1$&$.0\cdot 10^{-4}$&$4$&$.0$&$1$&$.5\cdot 10^{-9}$&$16$&$.0$&$6$&$.4\cdot 10^{-10}$&$16$&$.0$&$1$&$.7\cdot 10^{-12}$&$31$&$.7$&$6$&$.4\cdot 10^{-15}$&$32$&$.8$\\%
$512$&$$&$2$&$.6\cdot 10^{-5}$&$4$&$.0$&$9$&$.6\cdot 10^{-11}$&$16$&$.0$&$4$&$.0\cdot 10^{-11}$&$16$&$.0$&$5$&$.4\cdot 10^{-14}$&$31$&$.7$&$4$&$.7\cdot 10^{-15}$&--&\\%
$1\,024$&$$&$6$&$.4\cdot 10^{-6}$&$4$&$.0$&$6$&$.0\cdot 10^{-12}$&$16$&$.0$&$2$&$.5\cdot 10^{-12}$&$16$&$.0$&$2$&$.9\cdot 10^{-15}$&$18$&$.6$&$3$&$.6\cdot 10^{-15}$&--&\\\bottomrule %
\end {tabular}%
    \\[5mm]
\begin {tabular}{r<{\pgfplotstableresetcolortbloverhangright }@{}l<{\pgfplotstableresetcolortbloverhangleft }r<{\pgfplotstableresetcolortbloverhangright }@{}l<{\pgfplotstableresetcolortbloverhangleft }r<{\pgfplotstableresetcolortbloverhangright }@{}l<{\pgfplotstableresetcolortbloverhangleft }r<{\pgfplotstableresetcolortbloverhangright }@{}l<{\pgfplotstableresetcolortbloverhangleft }r<{\pgfplotstableresetcolortbloverhangright }@{}l<{\pgfplotstableresetcolortbloverhangleft }r<{\pgfplotstableresetcolortbloverhangright }@{}l<{\pgfplotstableresetcolortbloverhangleft }r<{\pgfplotstableresetcolortbloverhangright }@{}l<{\pgfplotstableresetcolortbloverhangleft }r<{\pgfplotstableresetcolortbloverhangright }@{}l<{\pgfplotstableresetcolortbloverhangleft }r<{\pgfplotstableresetcolortbloverhangright }@{}l<{\pgfplotstableresetcolortbloverhangleft }}%
\toprule \multicolumn {2}{c}{$K$}&\multicolumn {2}{c}{$n=6$}&\multicolumn {2}{c}{$R_{C}$}&\multicolumn {2}{c}{$n=7$}&\multicolumn {2}{c}{$R_{C}$}&\multicolumn {2}{c}{$n=8$}&\multicolumn {2}{c}{$R_{C}$}&\multicolumn {2}{c}{$n=9$}&\multicolumn {2}{c}{$R_{C}$}\\\midrule %
$2$&$$&$1$&$.6\cdot 10^{-3}$&--&&$1$&$.6\cdot 10^{-4}$&--&&$3$&$.2\cdot 10^{-5}$&--&&$2$&$.3\cdot 10^{-6}$&--&\\%
$4$&$$&$1$&$.1\cdot 10^{-5}$&$148$&$.2$&$1$&$.3\cdot 10^{-6}$&$129$&$.3$&$4$&$.8\cdot 10^{-8}$&$657$&$.8$&$4$&$.3\cdot 10^{-9}$&$521$&$.5$\\%
$8$&$$&$1$&$.1\cdot 10^{-7}$&$98$&$.0$&$5$&$.5\cdot 10^{-9}$&$229$&$.1$&$1$&$.3\cdot 10^{-10}$&$373$&$.8$&$5$&$.3\cdot 10^{-12}$&$817$&$.1$\\%
$16$&$$&$9$&$.6\cdot 10^{-10}$&$117$&$.1$&$2$&$.2\cdot 10^{-11}$&$254$&$.6$&$2$&$.8\cdot 10^{-13}$&$459$&$.0$&$5$&$.2\cdot 10^{-15}$&$1\,019$&$.3$\\%
$32$&$$&$7$&$.6\cdot 10^{-12}$&$125$&$.6$&$8$&$.8\cdot 10^{-14}$&$245$&$.2$&$4$&$.7\cdot 10^{-15}$&$60$&$.4$&$2$&$.0\cdot 10^{-15}$&--&\\%
$64$&$$&$6$&$.1\cdot 10^{-14}$&$125$&$.5$&$4$&$.7\cdot 10^{-15}$&$18$&$.8$&$3$&$.3\cdot 10^{-15}$&--&&$2$&$.0\cdot 10^{-15}$&--&\\%
$128$&$$&$2$&$.7\cdot 10^{-15}$&$22$&$.8$&$5$&$.3\cdot 10^{-15}$&--&&$4$&$.4\cdot 10^{-15}$&--&&$4$&$.4\cdot 10^{-15}$&--&\\%
$256$&$$&$2$&$.2\cdot 10^{-15}$&--&&$4$&$.2\cdot 10^{-15}$&--&&$4$&$.0\cdot 10^{-15}$&--&&$2$&$.2\cdot 10^{-15}$&--&\\%
$512$&$$&$2$&$.7\cdot 10^{-15}$&--&&$5$&$.3\cdot 10^{-15}$&--&&$4$&$.4\cdot 10^{-15}$&--&&$2$&$.7\cdot 10^{-15}$&--&\\%
$1\,024$&$$&$3$&$.1\cdot 10^{-15}$&--&&$4$&$.9\cdot 10^{-15}$&--&&$4$&$.4\cdot 10^{-15}$&--&&$2$&$.7\cdot 10^{-15}$&--&\\\bottomrule %
\end {tabular}%
}
\caption{\small{The 2D case. The errors in the mesh uniform norm and their ratios $R_C$ in dependence on $K=2,4,...,1024$ and $n=\overline{1,9}$ for algorithm (a)}
\label{tab:EX72:2D:FFT:ErrorCabs}}
\end{table}
\begin{figure}[htbp]\centering{
    \begin{minipage}[h]{0.49\linewidth}
        \center{\includegraphics[width=1\linewidth]{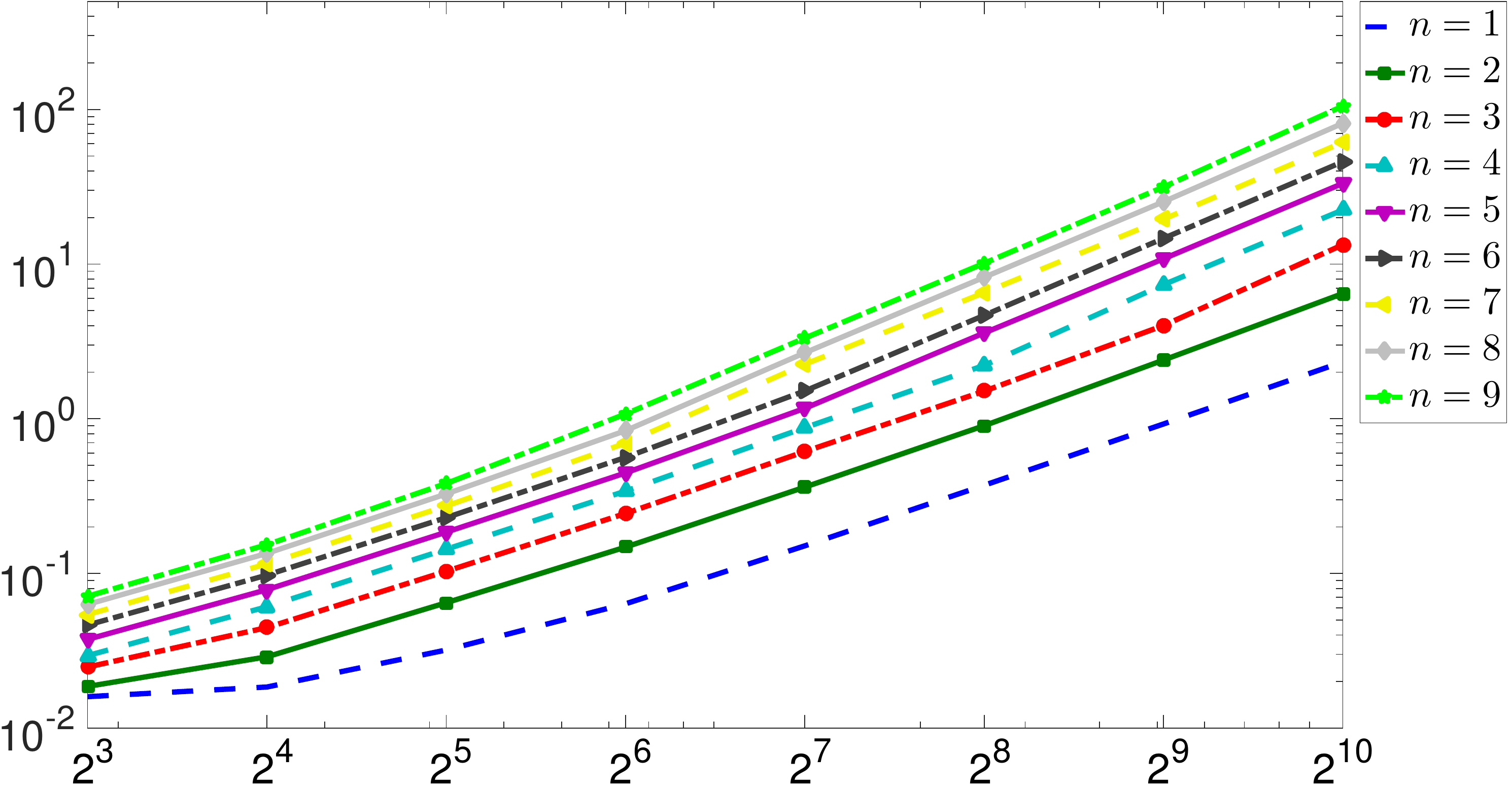}}\\ (a)
    \end{minipage}
    \begin{minipage}[h]{0.49\linewidth}
        \center{\includegraphics[width=1\linewidth]{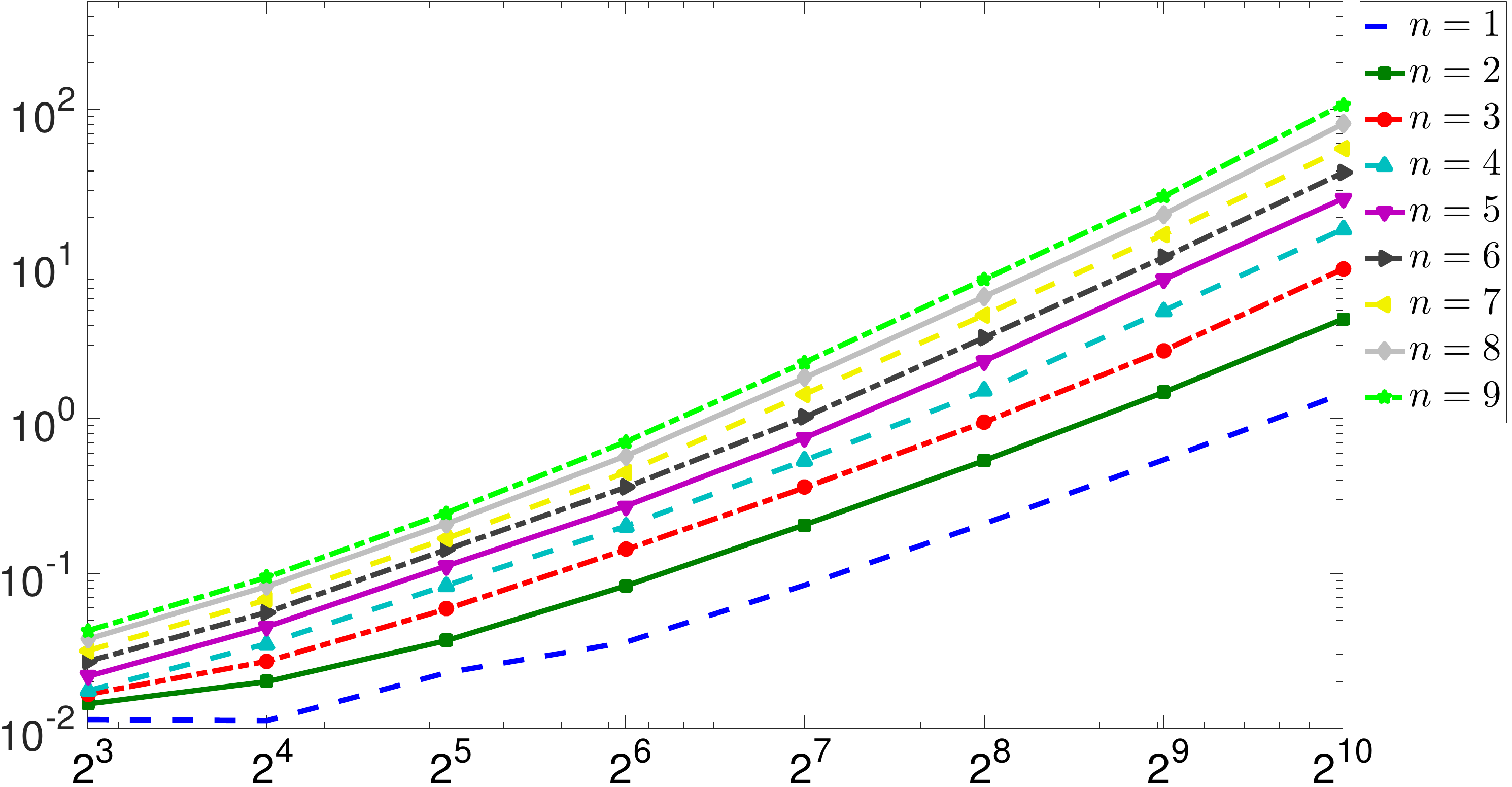}}\\ (b)
    \end{minipage}
    }\caption{\small{The 2D case. The execution time (in seconds) for algorithms (a) and (b) in dependence on $K=8,16,...,1024$ for $n=\overline{1,9}$}}
\label{fig:TIME:2D:Tsol}
\end{figure}
\begin{table}\centering{
\begin {tabular}{r<{\pgfplotstableresetcolortbloverhangright }@{}l<{\pgfplotstableresetcolortbloverhangleft }r<{\pgfplotstableresetcolortbloverhangright }@{}l<{\pgfplotstableresetcolortbloverhangleft }r<{\pgfplotstableresetcolortbloverhangright }@{}l<{\pgfplotstableresetcolortbloverhangleft }r<{\pgfplotstableresetcolortbloverhangright }@{}l<{\pgfplotstableresetcolortbloverhangleft }r<{\pgfplotstableresetcolortbloverhangright }@{}l<{\pgfplotstableresetcolortbloverhangleft }r<{\pgfplotstableresetcolortbloverhangright }@{}l<{\pgfplotstableresetcolortbloverhangleft }r<{\pgfplotstableresetcolortbloverhangright }@{}l<{\pgfplotstableresetcolortbloverhangleft }r<{\pgfplotstableresetcolortbloverhangright }@{}l<{\pgfplotstableresetcolortbloverhangleft }r<{\pgfplotstableresetcolortbloverhangright }@{}l<{\pgfplotstableresetcolortbloverhangleft }r<{\pgfplotstableresetcolortbloverhangright }@{}l<{\pgfplotstableresetcolortbloverhangleft }}%
\toprule \multicolumn {2}{c}{$K$}&\multicolumn {2}{c}{$n=1$}&\multicolumn {2}{c}{$n=2$}&\multicolumn {2}{c}{$n=3$}&\multicolumn {2}{c}{$n=4$}&\multicolumn {2}{c}{$n=5$}&\multicolumn {2}{c}{$n=6$}&\multicolumn {2}{c}{$n=7$}&\multicolumn {2}{c}{$n=8$}&\multicolumn {2}{c}{$n=9$}\\\midrule %
$8$&$$&--&&--&&--&&$1$&$.46$&$1$&$.6$&$1$&$.97$&$2$&$$&$2$&$.02$&$2$&$.01$\\%
$16$&$$&$1$&$.15$&$1$&$.55$&$1$&$.8$&$2$&$.07$&$2$&$.1$&$2$&$.1$&$2$&$.13$&$2$&$.14$&$2$&$.15$\\%
$32$&$$&$1$&$.74$&$2$&$.24$&$2$&$.31$&$2$&$.35$&$2$&$.35$&$2$&$.35$&$2$&$.37$&$2$&$.42$&$2$&$.49$\\%
$64$&$$&$1$&$.99$&$2$&$.29$&$2$&$.37$&$2$&$.38$&$2$&$.41$&$2$&$.45$&$2$&$.55$&$2$&$.57$&$2$&$.81$\\%
$128$&$$&$2$&$.37$&$2$&$.44$&$2$&$.51$&$2$&$.57$&$2$&$.62$&$2$&$.69$&$3$&$.24$&$3$&$.19$&$3$&$.1$\\%
$256$&$$&$2$&$.46$&$2$&$.49$&$2$&$.46$&$2$&$.53$&$3$&$.07$&$3$&$.08$&$2$&$.9$&$3$&$.08$&$3$&$.05$\\%
$512$&$$&$2$&$.49$&$2$&$.66$&$2$&$.66$&$3$&$.33$&$3$&$.02$&$3$&$.15$&$2$&$.99$&$3$&$.08$&$3$&$.11$\\%
$1\,024$&$$&$2$&$.52$&$2$&$.69$&$3$&$.33$&$3$&$.04$&$3$&$.07$&$3$&$.13$&$3$&$.14$&$3$&$.22$&$3$&$.34$\\\bottomrule %
\end {tabular}%
} \caption{\small{The 2D case.
The ratios of the execution times
for algorithm (a) in dependence on $K=2,4,...,1024$ and $n=\overline{1,9}$}
\label{tab:TIME:2D:FFT:Tsol}}
\end{table}
\begin{table}\centering{
\begin {tabular}{r<{\pgfplotstableresetcolortbloverhangright }@{}l<{\pgfplotstableresetcolortbloverhangleft }r<{\pgfplotstableresetcolortbloverhangright }@{}l<{\pgfplotstableresetcolortbloverhangleft }r<{\pgfplotstableresetcolortbloverhangright }@{}l<{\pgfplotstableresetcolortbloverhangleft }r<{\pgfplotstableresetcolortbloverhangright }@{}l<{\pgfplotstableresetcolortbloverhangleft }r<{\pgfplotstableresetcolortbloverhangright }@{}l<{\pgfplotstableresetcolortbloverhangleft }r<{\pgfplotstableresetcolortbloverhangright }@{}l<{\pgfplotstableresetcolortbloverhangleft }r<{\pgfplotstableresetcolortbloverhangright }@{}l<{\pgfplotstableresetcolortbloverhangleft }r<{\pgfplotstableresetcolortbloverhangright }@{}l<{\pgfplotstableresetcolortbloverhangleft }r<{\pgfplotstableresetcolortbloverhangright }@{}l<{\pgfplotstableresetcolortbloverhangleft }r<{\pgfplotstableresetcolortbloverhangright }@{}l<{\pgfplotstableresetcolortbloverhangleft }}%
\toprule \multicolumn {2}{c}{$K$}&\multicolumn {2}{c}{$n=1$}&\multicolumn {2}{c}{$n=2$}&\multicolumn {2}{c}{$n=3$}&\multicolumn {2}{c}{$n=4$}&\multicolumn {2}{c}{$n=5$}&\multicolumn {2}{c}{$n=6$}&\multicolumn {2}{c}{$n=7$}&\multicolumn {2}{c}{$n=8$}&\multicolumn {2}{c}{$n=9$}\\\midrule %
$8$&$$&--&&--&&--&&$1$&$.17$&$1$&$.28$&$1$&$.97$&$1$&$.99$&$1$&$.99$&$2$&$.01$\\%
$16$&&--&&$1$&$.39$&$1$&$.64$&$2$&$.02$&$2$&$.09$&$2$&$.08$&$2$&$.13$&$2$&$.2$&$2$&$.22$\\%
$32$&$$&$2$&$.04$&$1$&$.84$&$2$&$.17$&$2$&$.38$&$2$&$.47$&$2$&$.54$&$2$&$.49$&$2$&$.53$&$2$&$.61$\\%
$64$&$$&$1$&$.57$&$2$&$.26$&$2$&$.42$&$2$&$.42$&$2$&$.44$&$2$&$.52$&$2$&$.67$&$2$&$.74$&$2$&$.88$\\%
$128$&$$&$2$&$.34$&$2$&$.49$&$2$&$.52$&$2$&$.67$&$2$&$.77$&$2$&$.85$&$3$&$.21$&$3$&$.21$&$3$&$.27$\\%
$256$&$$&$2$&$.5$&$2$&$.6$&$2$&$.64$&$2$&$.8$&$3$&$.13$&$3$&$.27$&$3$&$.25$&$3$&$.34$&$3$&$.44$\\%
$512$&$$&$2$&$.58$&$2$&$.76$&$2$&$.91$&$3$&$.31$&$3$&$.37$&$3$&$.29$&$3$&$.31$&$3$&$.4$&$3$&$.43$\\%
$1\,024$&$$&$2$&$.68$&$2$&$.98$&$3$&$.36$&$3$&$.36$&$3$&$.33$&$3$&$.53$&$3$&$.58$&$3$&$.84$&$3$&$.94$\\\bottomrule %
\end {tabular}%
} \caption{\small{The 2D case.
The ratios of the execution times
for algorithm (b) in dependence on $K=2,4,...,1024$ and $n=\overline{1,9}$}
\label{tab:TIME:2D:CHOL:Tsol}}
\end{table}
\par Finally we consider the most interesting 3D case ($N=3$) and take the exact solution
$u(x):=\sin(2\pi x_1)\sin(3\pi x_2)$ $\sin(4\pi x_3)\cosh\big(\sqrt{2}x_1-x_2+x_3/\sqrt{3}\big)$.
Once again we compute the FEM solution for different values of $K=2,4,\ldots,64$ and $n=\overline{1,9}$ and study the error behaviour, see Fig. \ref{fig:EX72:3D:ErrorC} and Table \ref{tab:EX72:3D:FFT:ErrorCabs} (where values of $R_C$ less than 3 are omitted).
Conclusions are generally the same as in the 2D case.
Notice that now the worse stability properties of algorithm (b) are visible only for $n=8,9$ since much less maximal value of $K$ is taken.
\par The execution times in the 3D case are presented in Fig. \ref{fig:TIME:3D:Tsol} and Tables \ref{tab:TIME:3D:FFT:Tsol} and \ref{tab:TIME:3D:CHOL:Tsol}.
Once more conclusions are similar to the 2D case.
All the ratios in the both tables are notably less than the lower bound 8 for the theoretical ratios.
Importantly, for the maximal $K=64$ and $n=9$, system \eqref{eq:bvp glob 2} contains more than $190\cdot10^6$ unknowns but only less than 15 min is required to solve it that is the nice result (especially taking into account the Matlab implementation of loops).
\begin{figure}[htbp]\centering{
    \begin{minipage}[h]{0.49\linewidth}
        \center{\includegraphics[width=1\linewidth]{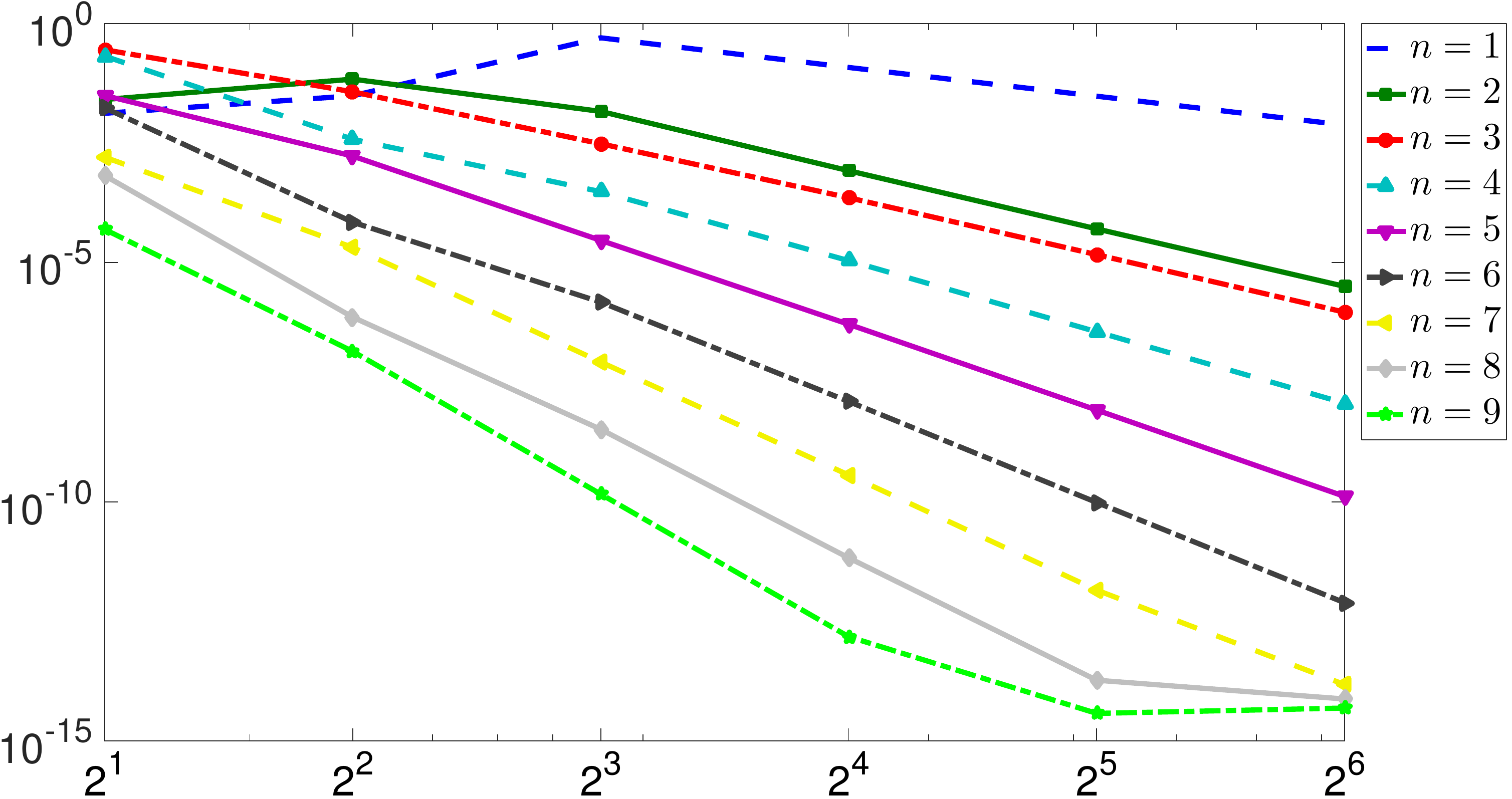}}\\ (a)
    \end{minipage}
    \begin{minipage}[h]{0.49\linewidth}
        \center{\includegraphics[width=1\linewidth]{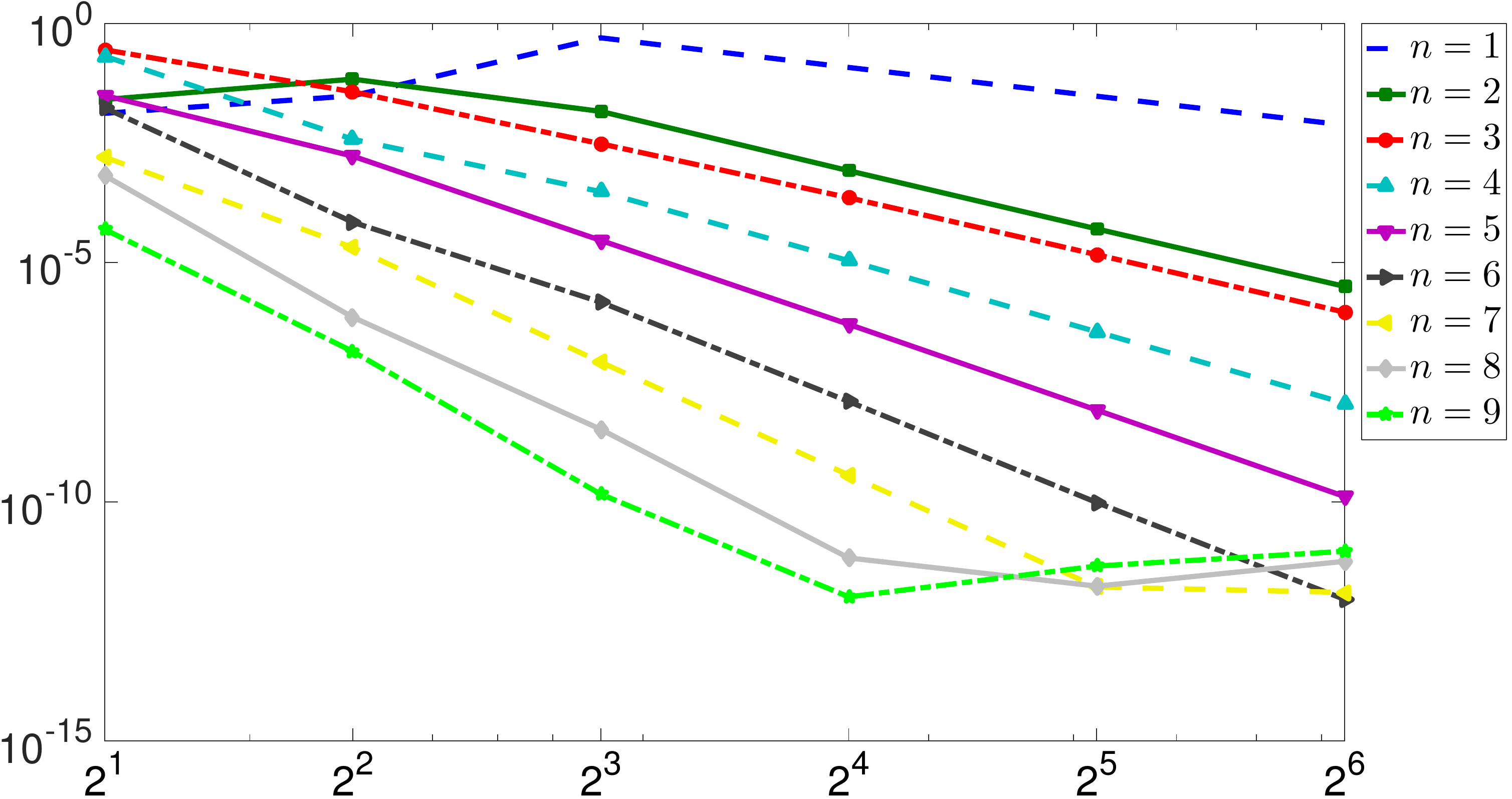}}\\ (b)
    \end{minipage}\\
    }\caption{\small{The 3D case. The errors in the mesh uniform norm for algorithms (a) and (b) in dependence on $K=2,4,...,64$ for $n=\overline{1,9}$}}
\label{fig:EX72:3D:ErrorC}
\end{figure}
\begin{table}\centering{
\begin {tabular}{r<{\pgfplotstableresetcolortbloverhangright }@{}l<{\pgfplotstableresetcolortbloverhangleft }r<{\pgfplotstableresetcolortbloverhangright }@{}l<{\pgfplotstableresetcolortbloverhangleft }r<{\pgfplotstableresetcolortbloverhangright }@{}l<{\pgfplotstableresetcolortbloverhangleft }r<{\pgfplotstableresetcolortbloverhangright }@{}l<{\pgfplotstableresetcolortbloverhangleft }r<{\pgfplotstableresetcolortbloverhangright }@{}l<{\pgfplotstableresetcolortbloverhangleft }r<{\pgfplotstableresetcolortbloverhangright }@{}l<{\pgfplotstableresetcolortbloverhangleft }r<{\pgfplotstableresetcolortbloverhangright }@{}l<{\pgfplotstableresetcolortbloverhangleft }r<{\pgfplotstableresetcolortbloverhangright }@{}l<{\pgfplotstableresetcolortbloverhangleft }r<{\pgfplotstableresetcolortbloverhangright }@{}l<{\pgfplotstableresetcolortbloverhangleft }r<{\pgfplotstableresetcolortbloverhangright }@{}l<{\pgfplotstableresetcolortbloverhangleft }r<{\pgfplotstableresetcolortbloverhangright }@{}l<{\pgfplotstableresetcolortbloverhangleft }}%
\toprule \multicolumn {2}{c}{$K$}&\multicolumn {2}{c}{$n=1$}&\multicolumn {2}{c}{$R_{C}$}&\multicolumn {2}{c}{$n=2$}&\multicolumn {2}{c}{$R_{C}$}&\multicolumn {2}{c}{$n=3$}&\multicolumn {2}{c}{$R_{C}$}&\multicolumn {2}{c}{$n=4$}&\multicolumn {2}{c}{$R_{C}$}&\multicolumn {2}{c}{$n=5$}&\multicolumn {2}{c}{$R_{C}$}\\\midrule %
$2$&$$&$1$&$.3\cdot 10^{-2}$&--&&$2$&$.6\cdot 10^{-2}$&--&&$2$&$.8\cdot 10^{-1}$&--&&$2$&$.1\cdot 10^{-1}$&--&&$3$&$.1\cdot 10^{-2}$&--&\\%
$4$&$$&$3$&$.1\cdot 10^{-2}$&--&&$6$&$.9\cdot 10^{-2}$&--&&$3$&$.7\cdot 10^{-2}$&$7$&$.6$&$3$&$.8\cdot 10^{-3}$&$54$&$.6$&$1$&$.7\cdot 10^{-3}$&$18$&$.5$\\%
$8$&$$&$5$&$.0\cdot 10^{-1}$&--&&$1$&$.5\cdot 10^{-2}$&$4$&$.7$&$3$&$.1\cdot 10^{-3}$&$12$&$.2$&$3$&$.0\cdot 10^{-4}$&$12$&$.5$&$2$&$.9\cdot 10^{-5}$&$57$&$.8$\\%
$16$&$$&$1$&$.2\cdot 10^{-1}$&$4$&$.2$&$8$&$.4\cdot 10^{-4}$&$17$&$.4$&$2$&$.3\cdot 10^{-4}$&$13$&$.4$&$1$&$.1\cdot 10^{-5}$&$27$&$.7$&$5$&$.1\cdot 10^{-7}$&$56$&$.8$\\%
$32$&$$&$3$&$.0\cdot 10^{-2}$&$4$&$.0$&$5$&$.1\cdot 10^{-5}$&$16$&$.4$&$1$&$.5\cdot 10^{-5}$&$15$&$.5$&$3$&$.6\cdot 10^{-7}$&$30$&$.6$&$8$&$.3\cdot 10^{-9}$&$62$&$.0$\\%
$64$&$$&$7$&$.5\cdot 10^{-3}$&$4$&$.0$&$3$&$.2\cdot 10^{-6}$&$16$&$.0$&$9$&$.2\cdot 10^{-7}$&$16$&$.0$&$1$&$.2\cdot 10^{-8}$&$31$&$.1$&$1$&$.3\cdot 10^{-10}$&$64$&$.0$\\\bottomrule %
\end {tabular}%
\\[5mm]
\begin {tabular}{r<{\pgfplotstableresetcolortbloverhangright }@{}l<{\pgfplotstableresetcolortbloverhangleft }r<{\pgfplotstableresetcolortbloverhangright }@{}l<{\pgfplotstableresetcolortbloverhangleft }r<{\pgfplotstableresetcolortbloverhangright }@{}l<{\pgfplotstableresetcolortbloverhangleft }r<{\pgfplotstableresetcolortbloverhangright }@{}l<{\pgfplotstableresetcolortbloverhangleft }r<{\pgfplotstableresetcolortbloverhangright }@{}l<{\pgfplotstableresetcolortbloverhangleft }r<{\pgfplotstableresetcolortbloverhangright }@{}l<{\pgfplotstableresetcolortbloverhangleft }r<{\pgfplotstableresetcolortbloverhangright }@{}l<{\pgfplotstableresetcolortbloverhangleft }r<{\pgfplotstableresetcolortbloverhangright }@{}l<{\pgfplotstableresetcolortbloverhangleft }r<{\pgfplotstableresetcolortbloverhangright }@{}l<{\pgfplotstableresetcolortbloverhangleft }}%
\toprule \multicolumn {2}{c}{$K$}&\multicolumn {2}{c}{$n=6$}&\multicolumn {2}{c}{$R_{C}$}&\multicolumn {2}{c}{$n=7$}&\multicolumn {2}{c}{$R_{C}$}&\multicolumn {2}{c}{$n=8$}&\multicolumn {2}{c}{$R_{C}$}&\multicolumn {2}{c}{$n=9$}&\multicolumn {2}{c}{$R_{C}$}\\\midrule %
$2$&$$&$1$&$.7\cdot 10^{-2}$&--&&$1$&$.6\cdot 10^{-3}$&--&&$6$&$.6\cdot 10^{-4}$&--&&$5$&$.0\cdot 10^{-5}$&--&\\%
$4$&$$&$7$&$.0\cdot 10^{-5}$&$245$&$.2$&$2$&$.1\cdot 10^{-5}$&$77$&$.3$&$7$&$.2\cdot 10^{-7}$&$909$&$.9$&$1$&$.4\cdot 10^{-7}$&$355$&$.4$\\%
$8$&$$&$1$&$.5\cdot 10^{-6}$&$46$&$.9$&$8$&$.4\cdot 10^{-8}$&$248$&$.1$&$3$&$.3\cdot 10^{-9}$&$221$&$.1$&$1$&$.4\cdot 10^{-10}$&$961$&$.4$\\%
$16$&$$&$1$&$.3\cdot 10^{-8}$&$119$&$.5$&$3$&$.6\cdot 10^{-10}$&$230$&$.7$&$6$&$.7\cdot 10^{-12}$&$486$&$.8$&$1$&$.5\cdot 10^{-13}$&$957$&$.3$\\%
$32$&$$&$9$&$.7\cdot 10^{-11}$&$128$&$.9$&$1$&$.4\cdot 10^{-12}$&$254$&$.2$&$1$&$.9\cdot 10^{-14}$&$360$&$.8$&$3$&$.8\cdot 10^{-15}$&$40$&$.1$\\%
$64$&$$&$7$&$.8\cdot 10^{-13}$&$124$&$.9$&$1$&$.5\cdot 10^{-14}$&$94$&$.6$&$7$&$.5\cdot 10^{-15}$&--&&$4$&$.9\cdot 10^{-15}$&--&\\\bottomrule %
\end {tabular}%
} \caption{\small{The 3D case. The errors in the mesh uniform norm and their ratios $R_C$ in dependence on $K=2,4,...,64$ and $n=\overline{1,9}$ for algorithm (a)}
\label{tab:EX72:3D:FFT:ErrorCabs}}
\end{table}
\begin{figure}[htbp]\centering{
    \begin{minipage}[h]{0.49\linewidth}
        \center{\includegraphics[width=1\linewidth]{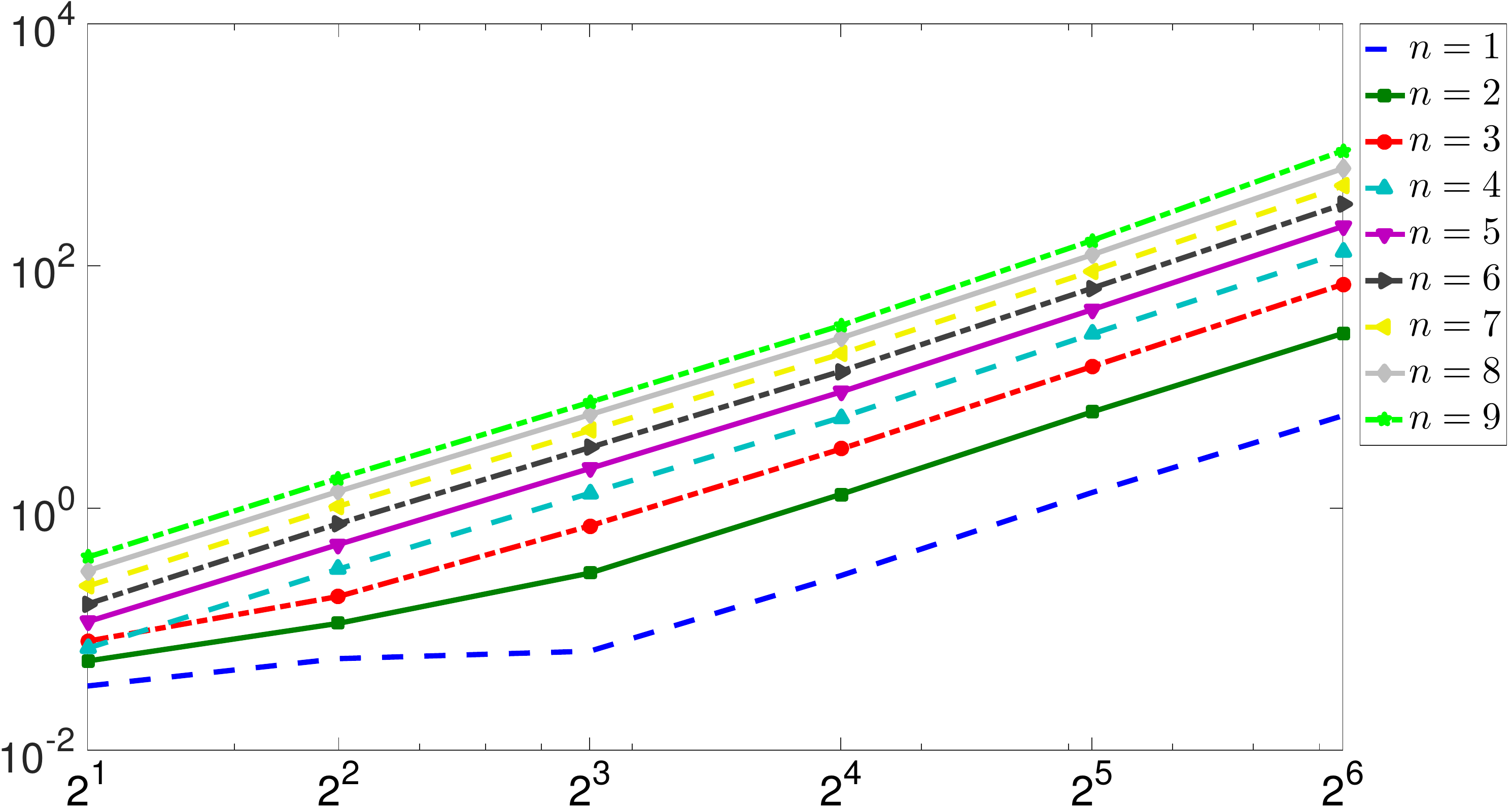}}\\ (a)
    \end{minipage}
    \begin{minipage}[h]{0.49\linewidth}
        \center{\includegraphics[width=1\linewidth]{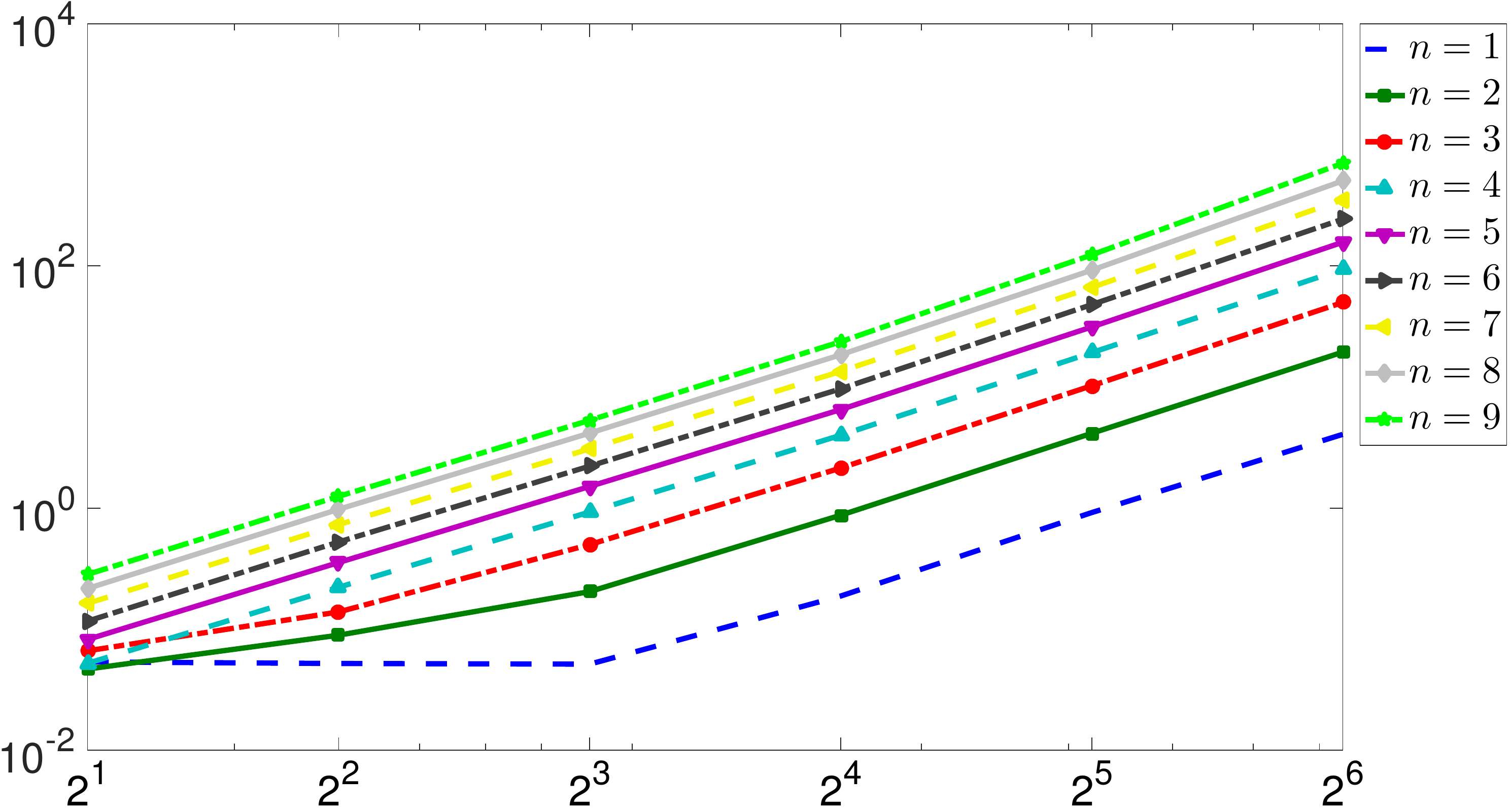}}\\ (b)
    \end{minipage}
    }\caption{\small{The 3D case. The execution time (in seconds) for algorithms (a) and (b) in dependence on $K=2,4,...,64$ for $n=\overline{1,9}$}}
\label{fig:TIME:3D:Tsol}
\end{figure}
\begin{table}\centering{
\begin {tabular}{r<{\pgfplotstableresetcolortbloverhangright }@{}l<{\pgfplotstableresetcolortbloverhangleft }r<{\pgfplotstableresetcolortbloverhangright }@{}l<{\pgfplotstableresetcolortbloverhangleft }r<{\pgfplotstableresetcolortbloverhangright }@{}l<{\pgfplotstableresetcolortbloverhangleft }r<{\pgfplotstableresetcolortbloverhangright }@{}l<{\pgfplotstableresetcolortbloverhangleft }r<{\pgfplotstableresetcolortbloverhangright }@{}l<{\pgfplotstableresetcolortbloverhangleft }r<{\pgfplotstableresetcolortbloverhangright }@{}l<{\pgfplotstableresetcolortbloverhangleft }r<{\pgfplotstableresetcolortbloverhangright }@{}l<{\pgfplotstableresetcolortbloverhangleft }r<{\pgfplotstableresetcolortbloverhangright }@{}l<{\pgfplotstableresetcolortbloverhangleft }r<{\pgfplotstableresetcolortbloverhangright }@{}l<{\pgfplotstableresetcolortbloverhangleft }r<{\pgfplotstableresetcolortbloverhangright }@{}l<{\pgfplotstableresetcolortbloverhangleft }}%
\toprule \multicolumn {2}{c}{$K$}&\multicolumn {2}{c}{$n=1$}&\multicolumn {2}{c}{$n=2$}&\multicolumn {2}{c}{$n=3$}&\multicolumn {2}{c}{$n=4$}&\multicolumn {2}{c}{$n=5$}&\multicolumn {2}{c}{$n=6$}&\multicolumn {2}{c}{$n=7$}&\multicolumn {2}{c}{$n=8$}&\multicolumn {2}{c}{$n=9$}\\\midrule %
$4$&$$&$1$&$.68$&$2$&$.05$&$2$&$.34$&$4$&$.51$&$4$&$.35$&$4$&$.62$&$4$&$.58$&$4$&$.52$&$4$&$.49$\\%
$8$&$$&$1$&$.15$&$2$&$.6$&$3$&$.82$&$4$&$.23$&$4$&$.23$&$4$&$.25$&$4$&$.22$&$4$&$.29$&$4$&$.24$\\%
$16$&$$&$4$&$.21$&$4$&$.47$&$4$&$.31$&$4$&$.22$&$4$&$.29$&$4$&$.27$&$4$&$.29$&$4$&$.3$&$4$&$.31$\\%
$32$&$$&$4$&$.85$&$4$&$.79$&$4$&$.8$&$4$&$.89$&$4$&$.78$&$4$&$.84$&$4$&$.83$&$4$&$.86$&$5$&$.03$\\%
$64$&$$&$4$&$.32$&$4$&$.46$&$4$&$.78$&$4$&$.77$&$4$&$.88$&$4$&$.93$&$5$&$.05$&$5$&$.16$&$5$&$.52$\\\bottomrule %
\end {tabular}%
} \caption{\small{The 3D case.
The ratios of the execution times
for algorithm (a) in dependence on $K=2,4,...,64$ and $n=\overline{1,9}$}
\label{tab:TIME:3D:FFT:Tsol}}
\end{table}
\begin{table}\centering{
\begin {tabular}{r<{\pgfplotstableresetcolortbloverhangright }@{}l<{\pgfplotstableresetcolortbloverhangleft }r<{\pgfplotstableresetcolortbloverhangright }@{}l<{\pgfplotstableresetcolortbloverhangleft }r<{\pgfplotstableresetcolortbloverhangright }@{}l<{\pgfplotstableresetcolortbloverhangleft }r<{\pgfplotstableresetcolortbloverhangright }@{}l<{\pgfplotstableresetcolortbloverhangleft }r<{\pgfplotstableresetcolortbloverhangright }@{}l<{\pgfplotstableresetcolortbloverhangleft }r<{\pgfplotstableresetcolortbloverhangright }@{}l<{\pgfplotstableresetcolortbloverhangleft }r<{\pgfplotstableresetcolortbloverhangright }@{}l<{\pgfplotstableresetcolortbloverhangleft }r<{\pgfplotstableresetcolortbloverhangright }@{}l<{\pgfplotstableresetcolortbloverhangleft }r<{\pgfplotstableresetcolortbloverhangright }@{}l<{\pgfplotstableresetcolortbloverhangleft }r<{\pgfplotstableresetcolortbloverhangright }@{}l<{\pgfplotstableresetcolortbloverhangleft }}%
\toprule \multicolumn {2}{c}{$K$}&\multicolumn {2}{c}{$n=1$}&\multicolumn {2}{c}{$n=2$}&\multicolumn {2}{c}{$n=3$}&\multicolumn {2}{c}{$n=4$}&\multicolumn {2}{c}{$n=5$}&\multicolumn {2}{c}{$n=6$}&\multicolumn {2}{c}{$n=7$}&\multicolumn {2}{c}{$n=8$}&\multicolumn {2}{c}{$n=9$}\\\midrule %
$4$&$$&--&&$1$&$.9$&$2$&$.09$&$4$&$.24$&$4$&$.32$&$4$&$.51$&$4$&$.47$&$4$&$.51$&$4$&$.41$\\%
$8$&$$&--&&$2$&$.29$&$3$&$.6$&$4$&$.19$&$4$&$.25$&$4$&$.23$&$4$&$.23$&$4$&$.26$&$4$&$.25$\\%
$16$&$$&$3$&$.65$&$4$&$.28$&$4$&$.29$&$4$&$.3$&$4$&$.32$&$4$&$.35$&$4$&$.36$&$4$&$.44$&$4$&$.46$\\%
$32$&$$&$4$&$.89$&$4$&$.75$&$4$&$.78$&$4$&$.81$&$4$&$.81$&$4$&$.96$&$4$&$.94$&$5$&$$&$5$&$.22$\\%
$64$&$$&$4$&$.43$&$4$&$.68$&$4$&$.87$&$4$&$.89$&$5$&$.01$&$5$&$.08$&$5$&$.31$&$5$&$.47$&$5$&$.71$\\\bottomrule %
\end {tabular}%
} \caption{\small{The 3D case.
The ratios of the execution times
for algorithm (b) in dependence on $K=2,4,...,64$ and $n=\overline{1,9}$}
\label{tab:TIME:3D:CHOL:Tsol}}
\end{table}

\medskip\par The study has been funded within the framework of the Academic Fund Program at the National Research University Higher School of Economics in 2016-2017 (grant no. 16-01-0054) and by the Russian Academic Excellence Project `5-100'
 as well as by the RFBR, grant №~16-01-00048.


\begin{thebibliography}{99}
\bibitem{A14}
Y.M. Altman, Accelerating MATLAB Performance: 1001 Tips to Speed up MATLAB Programs, Chapman and Hall/CRC, 2014.

\bibitem{BFK11}
B. Bialecki, G. Fairweather, A. Karageorghis, Matrix decomposition algorithms for elliptic boundary value problems: a survey, Numer. Algor. 56 (2011) 253–-295.

\bibitem{BRY07}
V. Britanak, K.R. Rao, P. Yip, Discrete cosine and sine transforms: general properties, fast algorithms and integer approximations. Oxford: Academic Press -- Elsevier, 2007.

\bibitem{C02}
P.G. Ciarlet, Finite Element Method for Elliptic Problems, SIAM, Philadelphia, 2002.

\bibitem{DFN09}
K. Du, G. Fairweather, Q.N. Nguyen, W. Sun,
Matrix decomposition algorithms for the $C^0$-quadratic finite element Galerkin method, BIT Numer. Math. 49 (2009) 509--526.

\bibitem{DZZ14}
B.~Ducomet, A.~Zlotnik, I.~Zlotnik,
The splitting in potential Crank-Nicolson scheme with discrete transparent boundary conditions for the Schr\"{o}dinger equation on a semi-infinite strip, ESAIM: Math. Model. Numer. Anal., 48:6 (2014), 1681-1699.

\bibitem{D96}
E.G. Dyakonov, Optimization in Solving Elliptic Problems, CRC Press, Boca Raton, 1996.

\bibitem{PrimeAnalysis}
S.~Eddins,
Timing the FFT, \verb"http://blogs.mathworks.com/steve/2014/04/07/timing-"
\verb"the-fft/".

\bibitem{FJ05}
M. Frigo, S.G. Johnson, The design and implementation of FFTW3,
Proc. IEEE 93 (2005) 216-231.

\bibitem{KS07}
Y.-Y. Kwan, J. Shen, An efficient direct parallel spectral-element solver for separable elliptic problems,
J. Comput. Phys. 225 (2007) 1721--1735.

\bibitem{SN78}
A.A. Samarskii, E.S. Nikolaev,
Numerical Methods for Grid Equations, Vol. I, Direct methods, Birkh\"{a}user, 1989.

\bibitem{S77}
P.N. Swarztrauber,
The methods of cyclic reduction, Fourier analysis and the FACR algorithm for the discrete solution of Poisson's equation on a rectangle,
SIAM Review {19}:3 (1977) 490--501.

\bibitem{ZZ12}	
{A.~Zlotnik, I.~Zlotnik},
{Finite element method with discrete transparent boundary conditions for the time-dependent 1D Schr\"o\-dinger equation},
Kinetic Relat. Models {5}:3 (2012), 639--667.

\bibitem{ZZDAN16}
A.A.~Zlotnik, I.A.~Zlotnik,
Fast direct algorithm for implementation of the high order finite element on rectangles for boundary value problems for the Poisson equation, Dokl. Math. (2017) (in press).

\bibitem{LTFAT}
\verb"http://ltfat.sourceforge.net"
\end{thebibliography}
\end{document}